\documentclass[a4paper,12pt]{article}
\usepackage{amsmath,amsfonts,amssymb,amsthm}
\usepackage{amssymb}
\usepackage{bbm}
\usepackage[dvips]{graphicx}
\usepackage{pdfsync}
\usepackage{color,xcolor}
\usepackage{verbatim}
\usepackage{enumerate}
\usepackage{authblk}

\newcommand{\set}[1]{\left\{#1\right\}}
\newcommand{\eps}{\varepsilon}

\renewcommand{\epsilon}{\varepsilon}

\renewcommand{\geq}{\geqslant}

\newtheorem{thm}{Theorem}[section]
\newtheorem*{thm*}{Theorem}

\newtheorem*{lem*}{Lemma}

\newtheorem*{cor*}{Corollary}

\newtheorem*{prop*}{Proposition}
\newtheorem{fct}[thm]{Fact}
\newtheorem*{fct*}{Fact}
\newtheorem{conj}[thm]{Conjecture}
\newtheorem*{conj*}{Conjecture}

\theoremstyle{definition}

\newtheorem*{defn*}{Definition}
\theoremstyle{remark}

\newtheorem*{rem*}{Remark}

\newtheorem*{example*}{Example}

\newtheorem*{que*}{Question}

\title{Empirical approach to the $\times 2,\times 3$ conjecture}

\author{Tomasz Downarowicz, Dawid Huczek}
\affil{\textit{Faculty of Pure and Applied Mathematics,
        Wroc{\l}aw University of Science and Technology, 
        Wybrze\.{z}e Wyspia\'{n}skiego 27,
        50-370 Wroc{\l}aw, Poland}\\
    \texttt{Tomasz.Downarowicz@pwr.edu.pl, Dawid.Huczek@pwr.edu.pl}}

\begin{document}

\maketitle

\begin{abstract}
We study atomic measures on $[0,1]$ which are invariant both under multiplication by $2\mod 1$ and by $3\mod 1$, since such measures play an important role in deciding Furstenberg's $\times 2, \times 3$ conjecture. Our specific focus was finding atomic measures whose supports are far from being uniformly distributed, and we used computer software to discover a number of such measures (which we call \emph{outlier measures}). The structure of these measures indicates the possibility that a sequence of atomic measures may converge to a non-Lebesgue measure; likely one which is a combination of the Lebesgue measure and one or more atomic measures. 
\end{abstract}

\section{Introduction}
Furstenberg's $\times 2, \times 3$ conjecture remains one of the major unsolved problems in the field of ergodic theory in dynamical systems. Originally posed in \cite{F}, it deals with the question of ergodic Borel probability measures on the circle, represented as the interval $I=[0,1]$ with the endpoints collapsed to one point, invariant under \emph{both} the actions $T(x)=2x\mod 1$ and $S(x)=3x \mod 1$.\footnote{The original conjecture deals more generally with $T(x)=px\mod 1$ and $S(x)=qx\mod1$ with $p,q$ being a pair of coprime natural numbers; the pair $(2,3)$ is the simplest such pair.} The Lebesgue measure $\lambda$ is obviously $T$ and $S$ invariant, and so is the point mass at $0$. There are also infinitely many finite orbits, minimal under the joint action of $T$ and $S$. Every such orbit consists of rational numbers with a denominator coprime with both $2$ and $3$ and each of them supports an atomic invariant measure---the normalized counting measure on the orbit. The open question is: are there any Borel measures on $I$ which are both $T$ and $S$ invariant, other than convex combinations of $\lambda$ with some atomic measures described above? An equivalent formulation reads: are there any continuous measures, invariant under both $T$ and $S$, and singular with respect to the Lebesgue measure $\lambda$? The original conjecture by Furstenberg is that the answer is negative, but a definite answer has not been reached so far. There are two major results in this direction; the first, proven by Furstenberg in the original paper \cite{F}, states that any nonatomic measure for the action of $T$ and $S$ has full support in $[0,1]$. Another result, due to Rudolph \cite{R}, states that for any Borel probability measure on $I$ which is both $T$ and $S$ invariant, and singular with respect to the Lebesgue measure, the actions of $T$ and $S$ both have entropy zero. Since then there has been no significant progress towards answering the question. Some mathematicians believe the conjecture is true, we are inclined to give no preference to any answer.

We believe that some valuable information can be acquired from studying the atomic measures supported by minimal finite orbits. There is a chance that a measure defying the conjecture exists among the accumulation points of such atomic measures. To the best of our knowledge, the structure of atomic measures has not been extensively investigated to date. In particular, it is unknown whether, with increasing number of atoms, the atomic measures must converge to the Lebesgue measure $\lambda$. This problem is weaker than the main Furstenberg conjecture, because a priori atomic measures could accumulate at an atomic measure or at a convex combination of $\lambda$ with some atomic measures, without defying the main conjecture. Nevertheless, it is an open question worth addressing. 
As a step in that direction, we have computationally generated a huge collection of atomic measures with large number of atoms (or more precisely, finite minimal orbits of large cardinality), focusing on any hints that the set of corresponding normalized counting measures might contain sequences converging to limits other than the Lebesgue measure $\lambda$. Interestingly, we have discovered a number of atomic invariant measures with many atoms, whose supports do not appear to be (approximately) uniformly distributed. We have also found out that it is very likely that such measures may converge to convex combinations of atomic measures possibly with some nonatomic part. The motivation for our experimental adventure was the hope to obtain some indication as to the existing types of accumulation points of the set of the atomic measures, and perhaps to shift the likelihood of either solution of Furstenberg's problem.

\section{Atomic invariant measures}
Throughout this and the following section we will consider the dynamical system consisting of the space $I=[0,1]$ with the identification $0=1$, under the $\mathbb Z^2$-action generated by two commuting transformations $T$ and $S$, defined by $T(x)=2x\mod 1$ and $S(x)=3x\mod 1$. In the subsequent discussion, whenever we use the word ``orbit'', or when we refer to concepts of invariance, minimality or ergodicity, we implicitly understand these concepts with respect to the above described $\mathbb Z^2$-action. Similarly, whenever we refer to ``atomic measures'', we refer exclusively to ergodic atomic measures (supported by minimal finite orbits). Any atomic measure is supported by the orbit of a single element $x\in[0,1]$, that is to say by the (finite) set of all numbers of the form $T^iS^j(x)$ for $i,j\in\mathbb{N}$ (note that the actions $S$ and $T$ commute, therefore the order of application does not matter), where $x$ is rational with a denominator $n$ not divisible by neither $2$ nor $3$. Conversely, the orbit of any such rational point is finite, minimal and supports an atomic measure.

The following questions can be asked, progressively:
\begin{enumerate}
\item\label{Q1} Can a sequence of atomic measures converge (in the weak-star topology) to a limit measure different from the Lebesgue measure?
\item\label{Q2} If yes, can this limit measure have any atoms?
\item\label{Q3} If yes, may the limit measure have a nontrivial continuous part?
\item\label{Q4} If yes, can the continuous part be different from the Lebesgue measure?
\end{enumerate}
Of course, positive answer to the last question would defy Furstenberg's conjecture and it is highly unlikely to be within reach. But even obtaining an answer to any of the questions (1)-(3) would be a remarkable progress on this subject.

We pose two other, independent questions:
\begin{enumerate}
\item[5.]\label{Q5} Does the closure of the set of atomic measures contain all ergodic measures?
\item[6.]\label{Q6} Is there a subset $K$ of atomic measures whose closure contains the 
convex hull of $K$?
\end{enumerate}
If Furstenberg's conjecture is true then the answer to question 5 is positive, while the converse implication does not hold. Nevertheless, even proving the positive answer seems out of reach. The last property is of special interest. If it held, then the closure of $K$ would be a closed convex subset with dense set of ergodic measures which are extreme in $K$. Although other extreme points of this subset need not be ergodic measures, nevertheless the existence of such a subset would strongly indicate that there are uncountably many ergodic measures.

Our main aim was to investigate questions 1--4, i.e. focus on the possible types of accumulation points of the set of atomic measures. Since any atomic measure is uniquely identified by its support, our direct object of study were the finite minimal orbits of the action generated by $S$ and $T$. We defer the detailed presentation of numeric results to the next section; here we will present a more theoretical discussion and general conclusions from our observations.

First of all, it is easy to see than any orbit has to be the set of numbers of the form $\frac{i}{n}$ where $n$ is a number relatively prime to both $2$ and $3$, and that $T$ and $S$ are both one-to-one on every orbit. This is why the only invariant measure supported by such an orbit is the normalized counting measure. We denote by $F_n$ the set of all proper fractions with denominator $n$, i.e., we let $F_n=\set{\frac{i}{n}:0<i<n}$. In many cases the orbit will be simply the set $F_n$, for example $F_5=\set{\frac{1}{5},\frac{2}{5},\frac{3}{5},\frac{4}{5}}$ is invariant under $T$ and $S$ and has no invariant subsets. If $n$ is not prime, then $F_n$ contains invariant sets $F_m$ for all divisors $m$ of $n$, and some (essential) remaining part $F_n'$. For example $F_{25}$ decomposes into the following orbits:
\[
O_1=\left\{\frac{1}{5},\frac{2}{5},\frac{3}{5},\frac{4}{5}\right\}.
\]
\[\begin{split}
O_2=\bigg\{&\frac{1}{25},\frac{2}{25},\frac{3}{25},\frac{4}{25},\frac{6}{25},\frac{7}{25},\frac{8}{25},\frac{9}{25},\frac{11}{25},\frac{12}{25},\\
&\frac{13}{25},\frac{14}{25},\frac{16}{25},\frac{17}{25},\frac{18}{25},\frac{19}{25},\frac{21}{25},\frac{22}{25},\frac{23}{25},\frac{24}{25}\bigg\},
\end{split}
\]

Because the sets $F_m$ are (nearly) uniformly distributed, we agree that $F_n'$ is also nearly uniformly distributed. That is to say, without further splitting into invariant sets, the normalized counting measures on the sets $F'_n$ do converge to the Lebesgue measure $\lambda$. However, $F_n'$ can split into a disjoint union of several orbits, and these orbits may reveal significant deviation from the uniform distribution. This may happen even if $n$ is a prime number (i.e., when $F_n'=F_n$).
The smallest such example is $n=23$, for which we obtain two orbits:
\[O_1=\set{\frac{1}{23},\frac{2}{23},\frac{3}{23},\frac{4}{23},\frac{6}{23},\frac{8}{23},\frac{9}{23},\frac{12}{23},\frac{13}{23},\frac{16}{23},\frac{18}{23}}\]
\[O_2=\set{\frac{5}{23},\frac{7}{23},\frac{10}{23},\frac{11}{23},\frac{14}{23},\frac{15}{23},\frac{17}{23},\frac{19}{23},\frac{20}{23},\frac{21}{23},\frac{22}{23}}\]
The sets $O_1$ and $O_2$ are both invariant under the actions of $T$ and $S$, therefore each of them supports an ergodic measure. These measures deviate from being (approximately) uniformly distributed: the one supported on $O_1$ is biased to the left, the one supported on $O_2$ is biased to the right. Also note the easy observation that $O_2$ consists of the numbers of the form $1-x$ where $x$ ranges over the elements of $O_1$. In general, if $O$ is an orbit under $T$ and $S$, then either an orbit is itself invariant under the transformation $f(x)=1-x$, or $f(O)$ is another orbit, disjoint with $O$. As $n$ increases, we observe that for some values of $n$ (not divisible by 2 or 3), the sets $F_n'$ decompose into an increasing number of orbits, even when $n$ is prime. 

The fact that $F_n'$ can be an union of several orbits indicates that there is at least the possibility that infinitely many of these orbits have normalized counting measures lying persistently far from $\lambda$, leading to non-Lebesque limit measures. By necessity however such measures are ``rare'', regardless of whether the Furstenberg conjecture is true:

\begin{thm}\label{thm:feworbits}
    For any $\delta>0$ and $\eps>0$ there exists an $N$ such that for every $n>N$, if $F_n$ decomposes into orbits $O_1,\ldots,O_k$, then  the cardinality of the union of the orbits corresponding to measures distant from the Lebesgue measure by at least $\delta$ is less than $\eps n$.
\end{thm}
\begin{proof}
	Let $K_\delta$ denote the set of all invariant measures whose distance from $\lambda$ is at least $\delta$. Let $n_1,\ldots,n_k$ denote the cardinality of $O_1,\ldots,O_k$ respectively, and let $\mu_1,\ldots,\mu_k$ denote the normalized counting measures on these orbits. If for arbitrarily large $n$ the cardinality of the union of the orbits corresponding to measures belonging to $K_\delta$ was larger than $\eps n$, then the convex combination $\nu_n=\sum_{i=1}^{k}\frac{n_i}{n-1}\mu_i$ could be written as $\nu_n=\alpha_n\nu_n'+\beta_n\nu_n''$, where $\nu_n'$ is a convex combination of the measures from $K_\delta$ and $\alpha_n\geq\eps$. Therefore we could find a subsequence $(n_k)$ such that $(\alpha_{n_k}\nu_{n_k}')$ converges to some measure $\alpha\nu'$, where $\alpha\geq\eps$ and $\nu'$ is not the Lebesgue measure $\lambda$ ($K_\delta$ is a compact set not containing the extreme point $\lambda$, therefore the closed convex envelope of $K_\delta$ is also a compact set not containing $\lambda$). Since $\nu_{n_k}$ is the uniform distribution on $F_{n_k}$, the sequence $(\nu_{n_k})$ converges to $\lambda$, and thus we get a decomposition $\lambda=\alpha\nu'+\beta\nu''$ where $\alpha>0$ and $\nu'\neq\lambda$, which contradicts the fact that $\lambda$ is an extreme point in the set of invariant measures.
\end{proof}

This means that even if the set $F_n$ for large $n$ decomposes into a large number of orbits, the majority of points in $F_n$ belong to orbits corresponding to measures lying close to the Lebesgue measure $\lambda$. We will focus specifically on finding the uncommon atomic invariant measures which are not close to $\lambda$, and which we will refer to as the \emph{outliers}.

Another observation concerning the outliers is that if a sequence of outliers converges to a measure which has atoms (and the outliers we have found do seem to exhibit concentration of mass around periodic orbits of lower periods), then another sequence of outliers converges to a measure with just one atom---the atom at zero.

\begin{fct}\label{fct:atom_at_one}
    If there exists a sequence of atomic invariant measures $(\mu_n)$ converging to a limit measure $\mu$ which has atoms of total mass $M$, then there exists a sequence of atomic invariant measures $(\nu_n)$ converging to an invariant measure $\nu$ such that $\nu(\set{0})=M$.
\end{fct}
\begin{proof}
    Fix a decreasing to zero sequence of positive numbers $(\eps_n)$. Since $\mu$ has at most countably many atoms, say $x_1,x_2,\ldots$, then $\epsilon_k=M-\mu(\{x_1,x_2,\dots,x_k\})$ is a sequence converging to zero.
Let $Q_k$ be the lowest common denominator of $\set{x_1,x_2,\ldots,x_k}$. Since the map $\psi_k(x)=Q_kx\mod 1$ is an endomorphism of the dynamical system $([0,1],S,T)$, the measure $\nu_k$ given by $\nu_k(A)=\mu(\psi_k^{-1}(A))$ is also an invariant measure, and since $\psi_k^{-1}(\set{0})$ contains $\set{x_1,x_2,\ldots,x_k}$, we have $\nu_k(\set{0})>M-\eps_k$. Furthermore, if we set $\nu_{k,n}(A)=\mu_n(\psi_k^{-1}(A))$, we will obtain a sequence of atomic invariant measures converging in $n$ to $\nu_k$. Using a diagonal argument, we can now find a sequence of atomic measures $\nu_{k,n_k}$ converging in $k$ to $\nu$.
\end{proof}

\section{Methodology}
We used the Mathematica software to decompose the sets $F_n'$ for $n$ up to $700,000$ (choosing only $n$ coprime with both $2$ and $3$) into disjoint orbits, using the following algorithm for every $n$ (note that we decompose $F_n'$ rather than all of $F_n$, in order to avoid generating duplicate orbits): 
\begin{enumerate}
    \item Let $k=1$, let $x=\frac kn$ and let $O_{n,k}=\set{x}$.
    \item Repeatedly add to $O_{n,k}$ all numbers of the form $2x \mod 1$ and $3x\mod 1$ for $x\in O_{n,k}$, until no new numbers are added.
    \item If $\bigcup_{i=1}^{k}O_{n,i}=F_n'$, then proceed to the following $n$; otherwise increase $k$ repeatedly by $1$ until $\frac kn$ becomes the smallest element of $F_n'$ which does not belong to $\bigcup_{i=1}^{k-1}O_{n,i}$. Then let $x=\frac kn$, $O_{n,k}=\{x\}$, and return to step $2$.
\end{enumerate}

We then selected the orbits for which the cumulative distribution function (cdf) for the normalized counting measure supported by the orbit differed from the cdf of the Lebesgue measure $\lambda$ on $[0,1]$ by at least $\frac{1}{10}$ at at least one point. The choice of threshold value $\frac1{10}$ is dictated by analysis of the histograms of the distance from $\lambda$ for individual denominators (these histograms are provided further below); in most cases, just above $\frac1{10}$, there appears a small isolated cluster. The main observation is that such ``outlier orbits'' have appeared for $n$ as large as $609,427$. As expected from theorem \ref{thm:feworbits}, the outlier orbits are rare and relatively short compared to the size of $F_n$ (which consists of $n-1$ points): for $n=609,427$ there are two outlier orbits, both of length $1,080$. In general, for denominators between $10,000$ and $700,000$ we found 64 outlier orbits. As none of them were symmetric around $\frac{1}{2}$, they formed 32 mirrored pairs, hence we chose just one of each pair for subsequent study (we chose the orbit which had more points in the left half of $[0,1]$ than in the right), ending up with 32 outlier orbits.  

Figure on the page \pageref{bigorbit} contains the cdf and histogram of the longest outlier orbit we have found---the orbit of the point $\frac{31}{609427}$. The corresponding normalized counting measure appears to be far from uniform, with a notable concentration of mass to the right of $1$, but also to the right of $\frac{1}{2}$ and several other rational points with small denominators of the form $2^p$ and $3^q$. Similar phenomena can be seen on pages \pageref{moreorbits_start}-\pageref{moreorbits_end} for other orbits (for smaller $n$). Later we will explain why a concentration of mass near $0$ forces smaller ``shadow concentrations'' of mass near $\frac12, \frac13, \frac23$ and so on. It can be predicted that if we classify some number of points, say $M$, in the orbit as close to zero, then, by the same criteria of closeness, there will be approximately $\sqrt{M\frac{2\ln2}{\ln3}}$ points close to $\frac12$ and $\frac12\sqrt{M\frac{2\ln3}{\ln2}}$ points close to each of the points $\frac13$ and $\frac23$. It can be thus seen that with increasing $M$ the contribution of the shadow concentrations will eventually vanish.

\subsection{Hypotheses}
Based on our experiments, we are inclined to believe that there are arbitrarily long outlier orbits (which exist within $F_n$ for arbitrarily large $n$). By passing to a convergent subsequence of the corresponding normalized counting measures, we would obtain a limit measure different from the Lebesgue measure $\lambda$. Our results also suggest that (at least some of) the limit measures have a positive atomic part. Fact \ref{fct:atom_at_one} implies that then there would also exist a sequence of outlier measures converging to a measure with a positive mass at $0$ and otherwise continuous. Such behavior is also observed in our numerical experiments. This motivates us to formulate the following conjecture (and call for a rigorous proof):

\begin{conj}
There exists a sequence of atomic ergodic measures converging to a measure with a positive atomic part.
\end{conj}

Moreover, we are inclined to believe that if the conjecture holds then every atomic measure may appear as part of some limit measure. 
\smallskip

Assuming Furstenberg's conjecture is true, the continuous part of every limit measure would have to be (proportional to) $\lambda$. But because in any finite resolution there is no difference between the histogram of a singular measure and a measure with many small atoms, we doubt if experimental research can provide any hint toward answering this question.

\section{Presentation of results}
We begin with the histogram and CDF of the counting measure on the orbit of $\frac{31}{609427}$, which consists of $1080$ points. The histogram, like all other histograms of orbits, has been normalized to represent a probability distribution (i.e. the total area of all the bars is $1$). Note the concentration of mass around $1$, as well as a smaller concentrations 
around $\frac{1}{2},\frac{1}{3},\frac{2}{3},\frac{1}{4},$ and $\frac{3}{4}$. We will comment on these further on page \pageref{fig:SymbolicLongest}.\\
\smallskip\label{bigorbit}
    \includegraphics[width=0.5\textwidth]{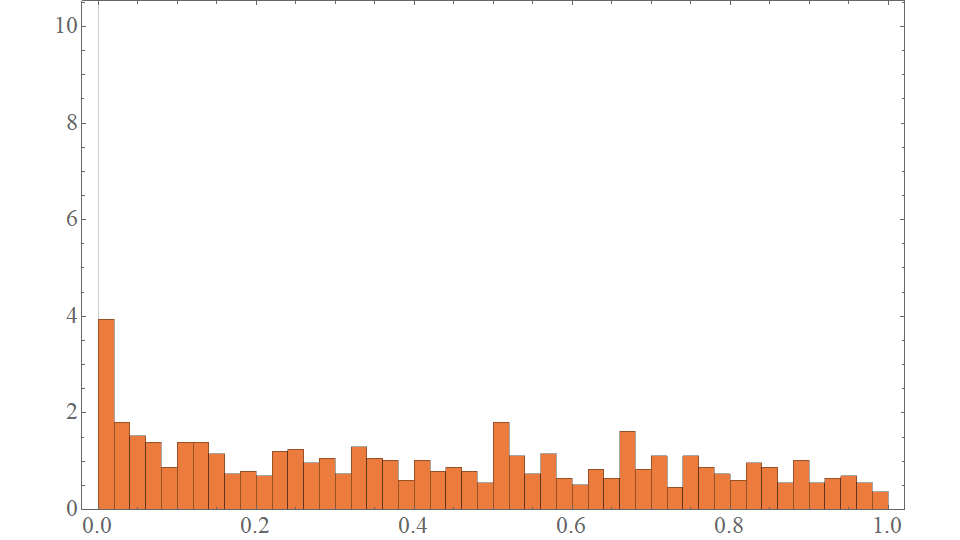}
    \includegraphics[width=0.5\textwidth]{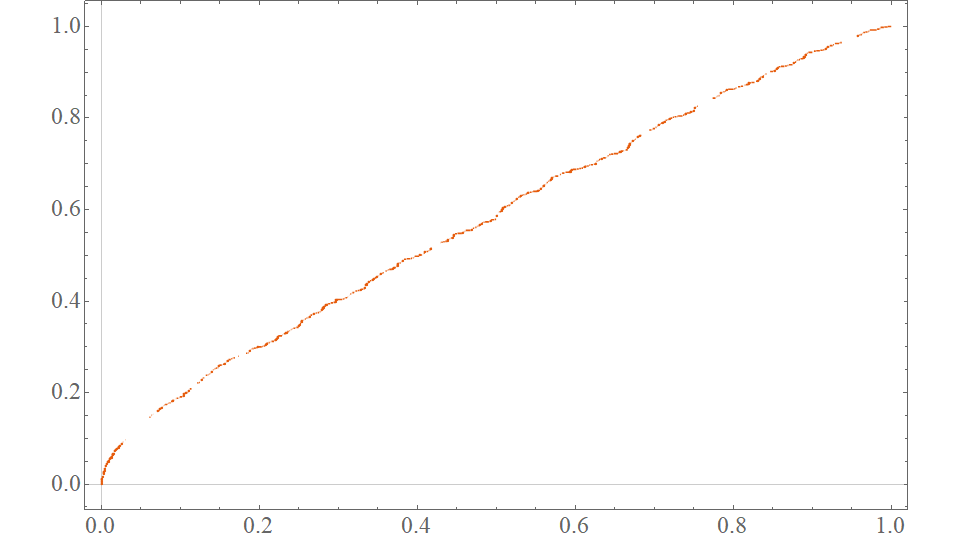}
    \pagebreak
The following 3 images present the histograms and CDF's of three orbits in $F_n$ for $n=609427$: the first one is the outlier orbit shown on previous page, the second one is the orbit whose corresponding atomic measure is the closest to $\lambda$ (it is the orbit of the point $\frac{461}{609427}$), and the last one is the orbit whose distance to $\lambda$ is the median among all the orbits within $F_{609427}$, namely the orbit of $\frac{2809}{609427}$. Each of the three orbits consists of 1080 points.
\smallskip\hrule\smallskip  \includegraphics[width=0.5\textwidth]{Histogram032.png} \includegraphics[width=0.5\textwidth]{CDF032.png}\\The histogram and CDF for the orbit of $\frac{31}{609427}$.
\smallskip\hrule\smallskip  \includegraphics[width=0.5\textwidth]{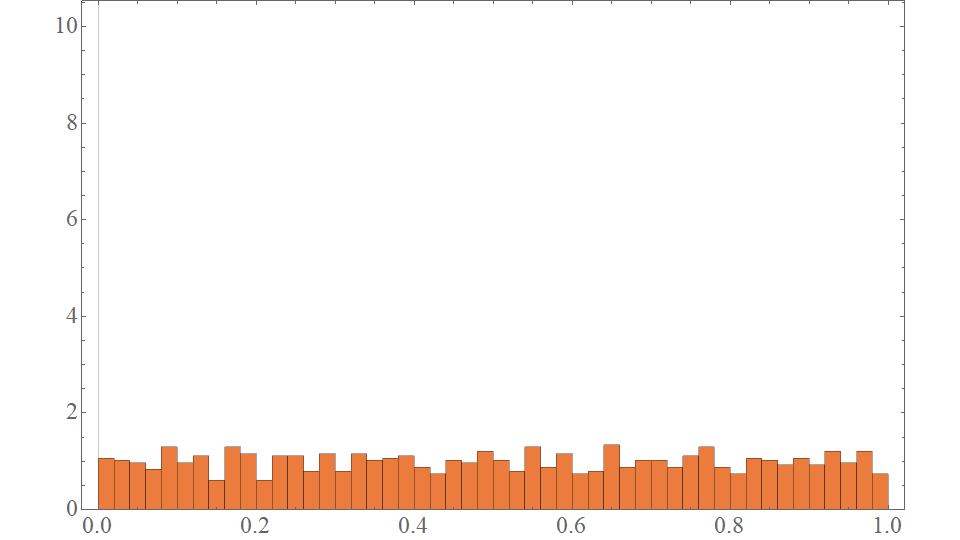} \includegraphics[width=0.5\textwidth]{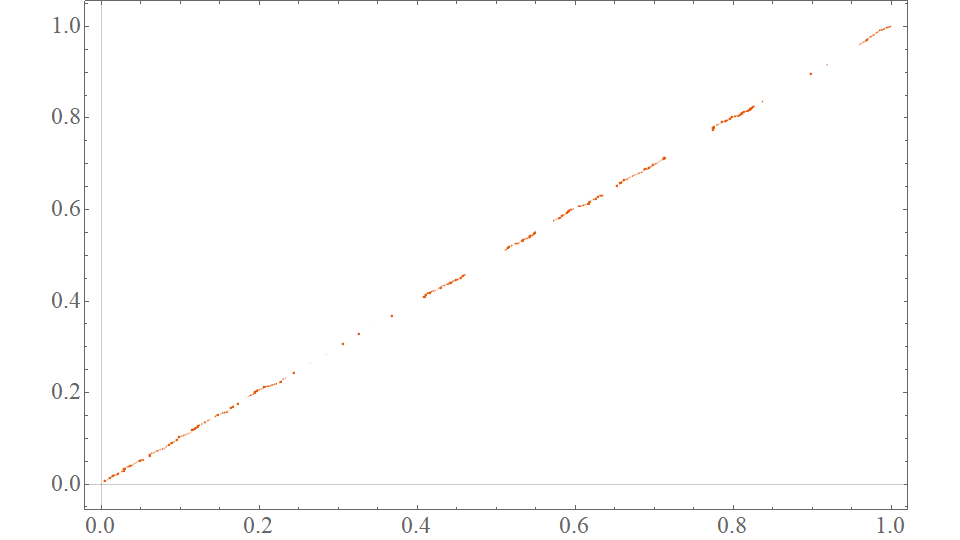}\\The histogram and CDF for the orbit of $\frac{461}{609427}$.
\smallskip\hrule\smallskip  \includegraphics[width=0.5\textwidth]{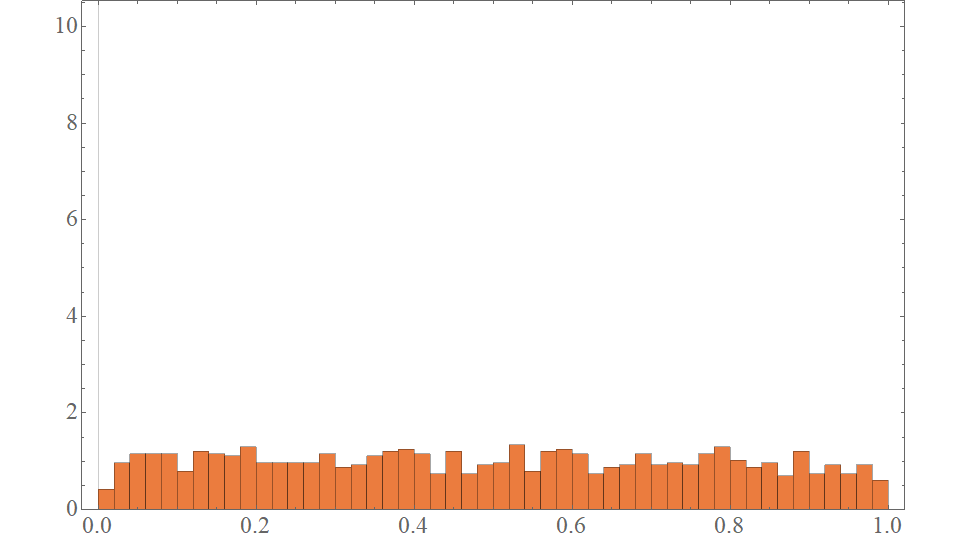} \includegraphics[width=0.5\textwidth]{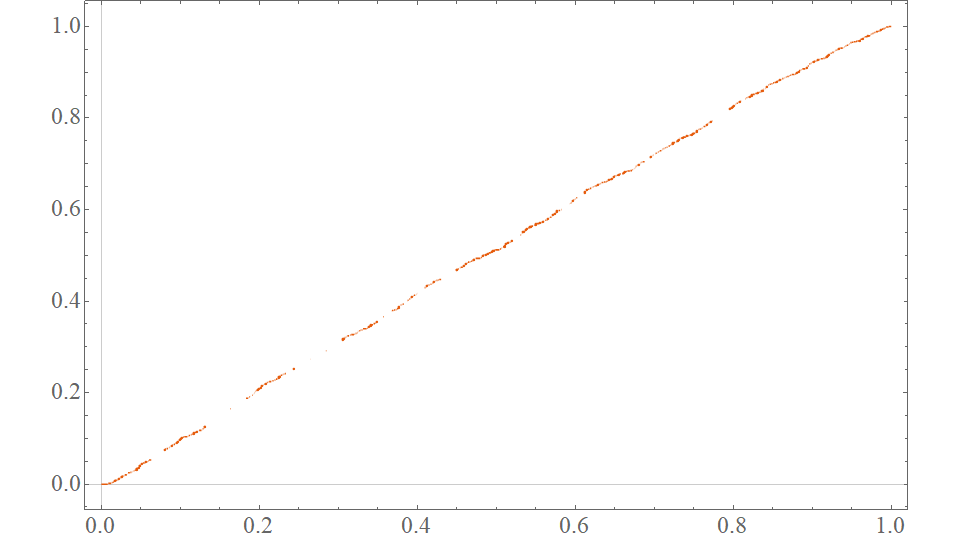}\\The histogram and CDF for the orbit of $\frac{5087}{609427}$.

\pagebreak
To justify our choice of outliers, we present the histograms of the distances between orbits of atomic measures for various denominators and the Lebesgue measure $\lambda$. We limited the histograms to include orbits of length $100$ or more, which are supported by points of the form $\frac{k}{n}$ where $n$ is the respective denominator and $k$ is coprime with $n$, and have more points in the left side of the interval $[0,1]$ than in the right (in fact for the denominators for which we discovered the outliers, no orbits were symmetric, so choosing just one orbit out of every mirrored pair simply removes redundancy from the data).

\noindent\includegraphics[width=0.5\textwidth]{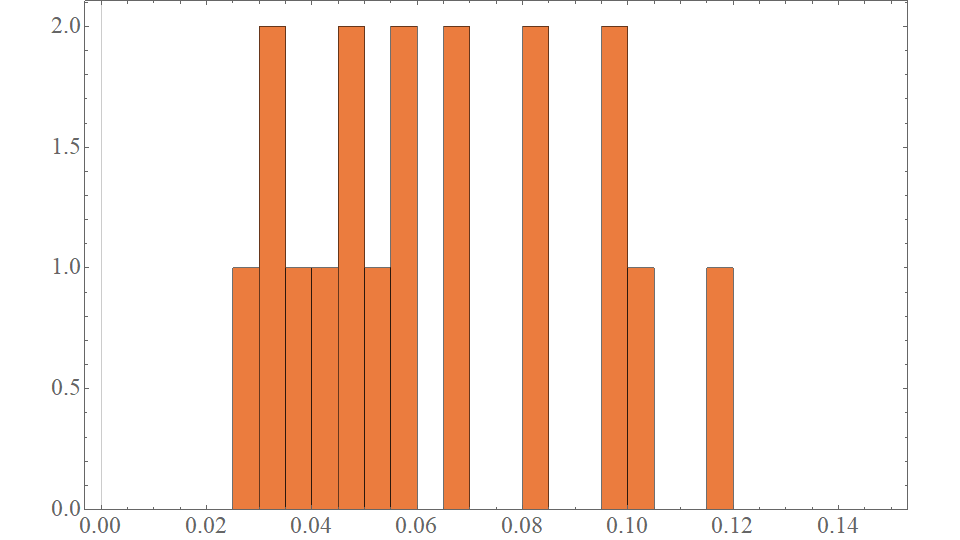} \includegraphics[width=0.5\textwidth]{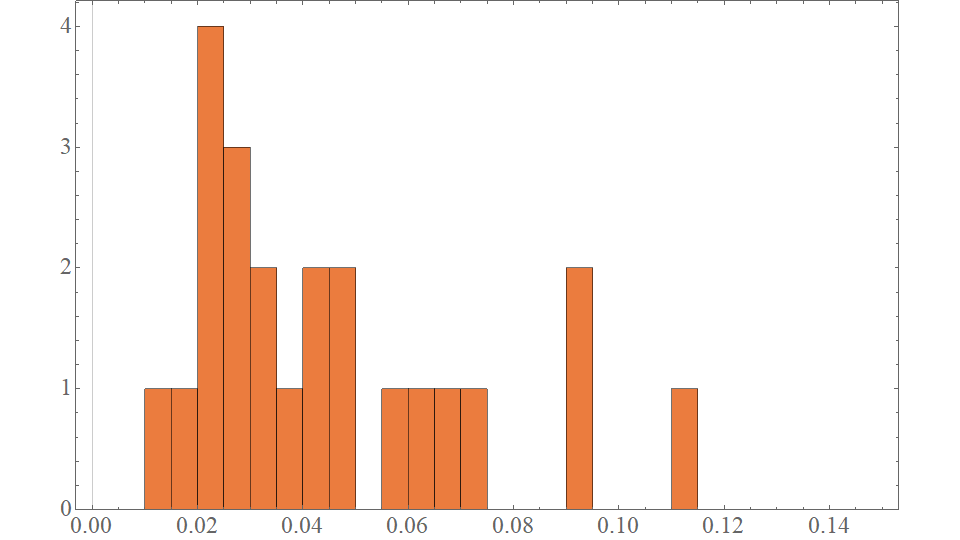}\\Histogram of the distance from $\lambda$ for a selected subset of most significant orbits in $F_{15025}$ (18 orbits) and $F_{21667}$(23 orbits).\hrule\smallskip\noindent\includegraphics[width=0.5\textwidth]{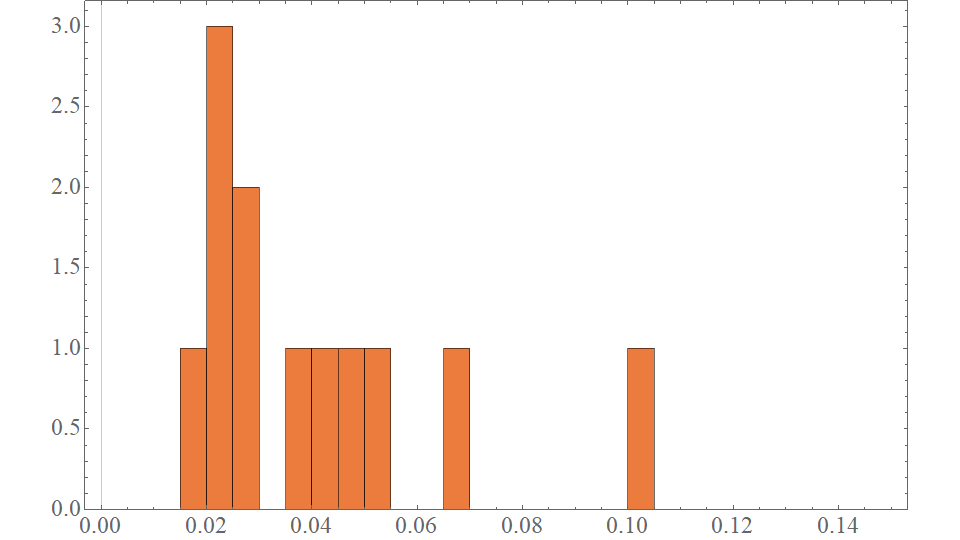} \includegraphics[width=0.5\textwidth]{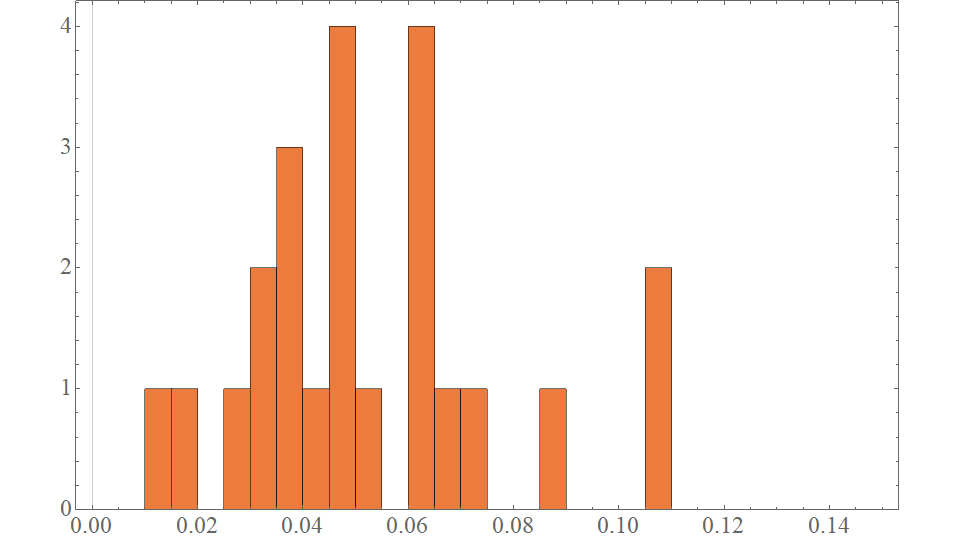}\\Histogram of the distance from $\lambda$ for a selected subset of most significant orbits in $F_{22015}$ (12 orbits) and $F_{33215}$(23 orbits).\hrule\smallskip\noindent\includegraphics[width=0.5\textwidth]{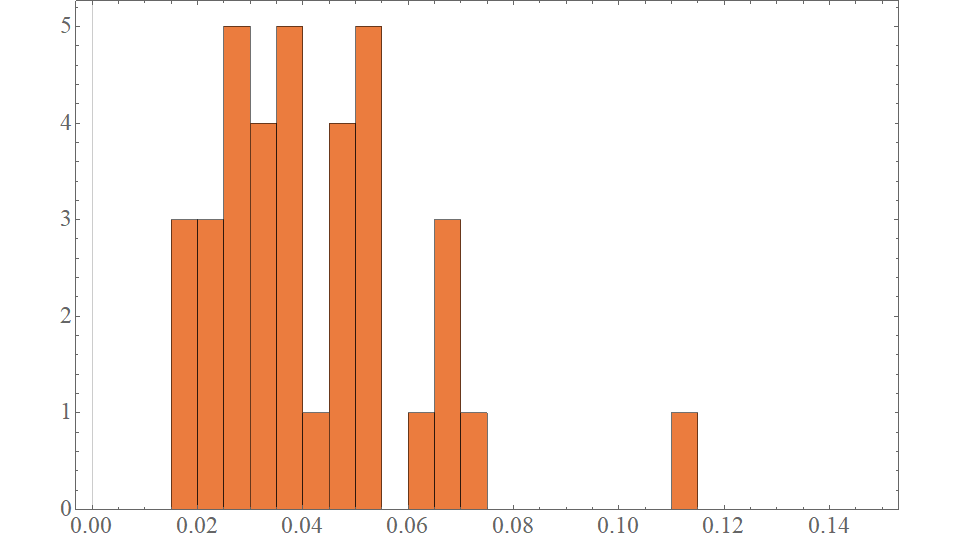} \includegraphics[width=0.5\textwidth]{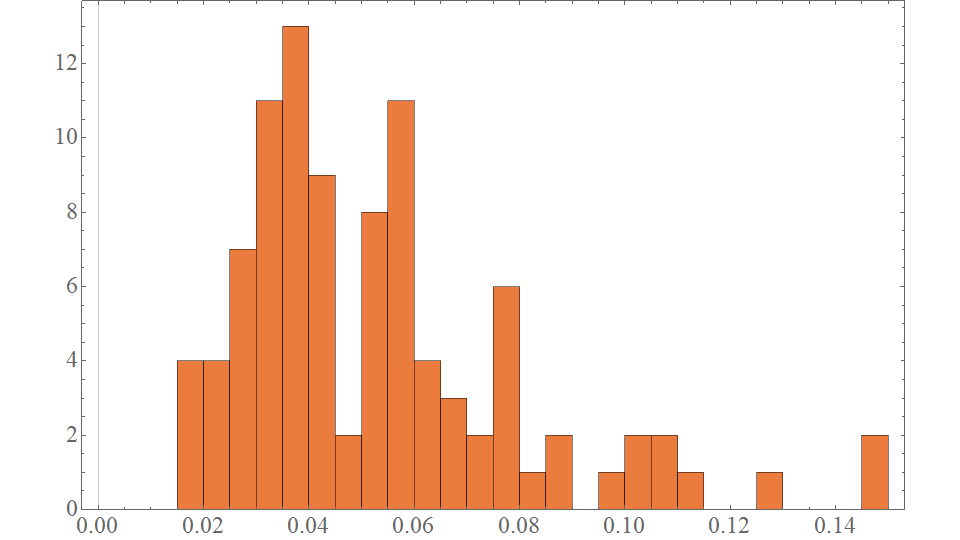}\\Histogram of the distance from $\lambda$ for a selected subset of most significant orbits in $F_{50557}$ (36 orbits) and $F_{86963}$(96 orbits).\hrule\smallskip\noindent\includegraphics[width=0.5\textwidth]{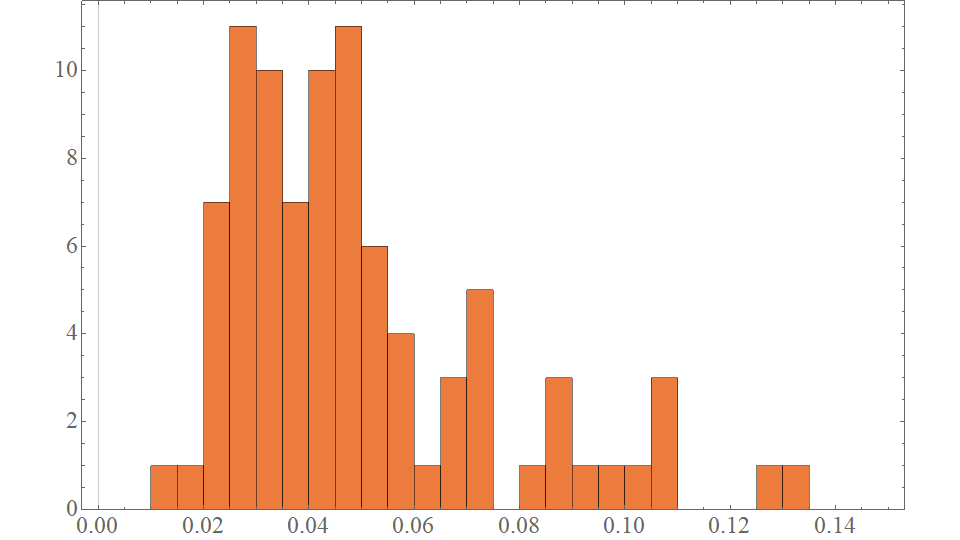} \includegraphics[width=0.5\textwidth]{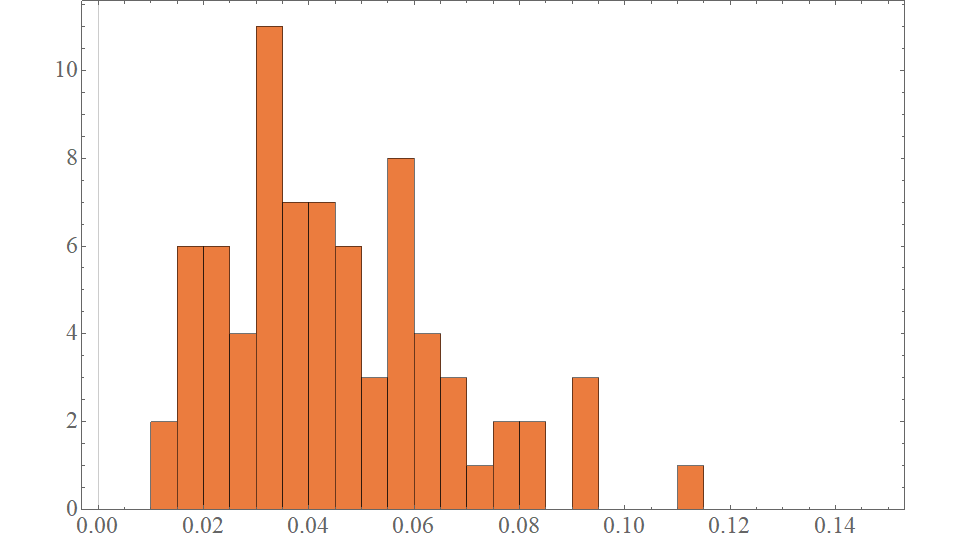}\\Histogram of the distance from $\lambda$ for a selected subset of most significant orbits in $F_{108335}$ (89 orbits) and $F_{116123}$(76 orbits).\hrule\smallskip\noindent\includegraphics[width=0.5\textwidth]{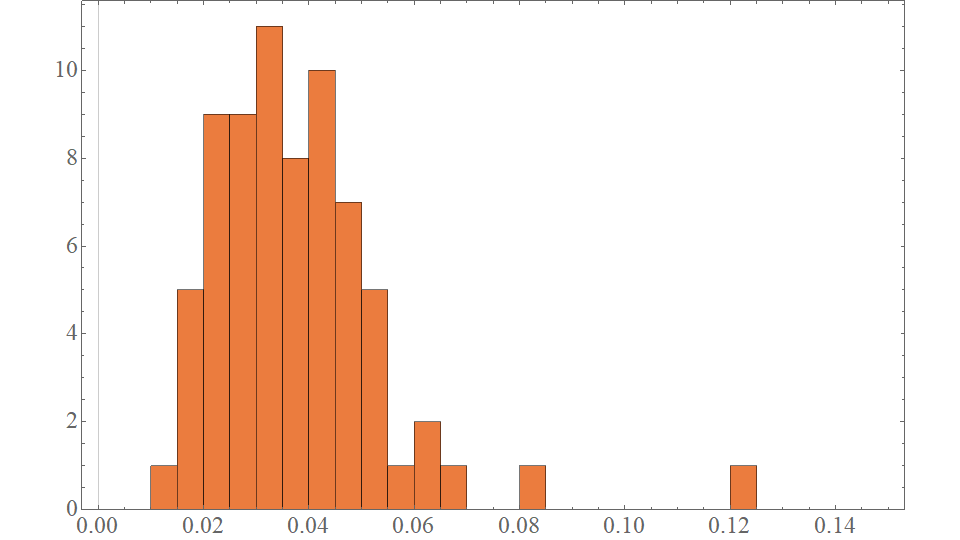} \includegraphics[width=0.5\textwidth]{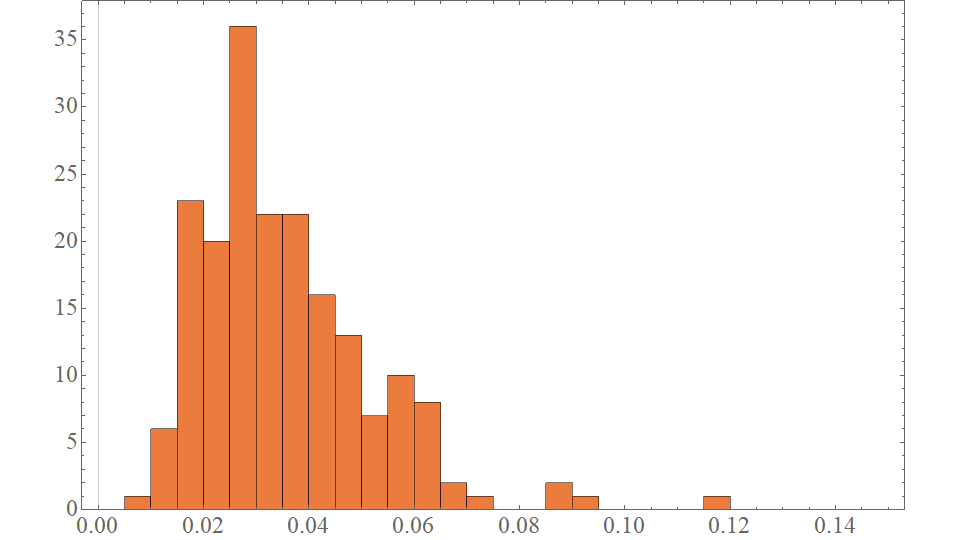}\\Histogram of the distance from $\lambda$ for a selected subset of most significant orbits in $F_{119795}$ (71 orbits) and $F_{434815}$(191 orbits).\hrule\smallskip\noindent\includegraphics[width=0.5\textwidth]{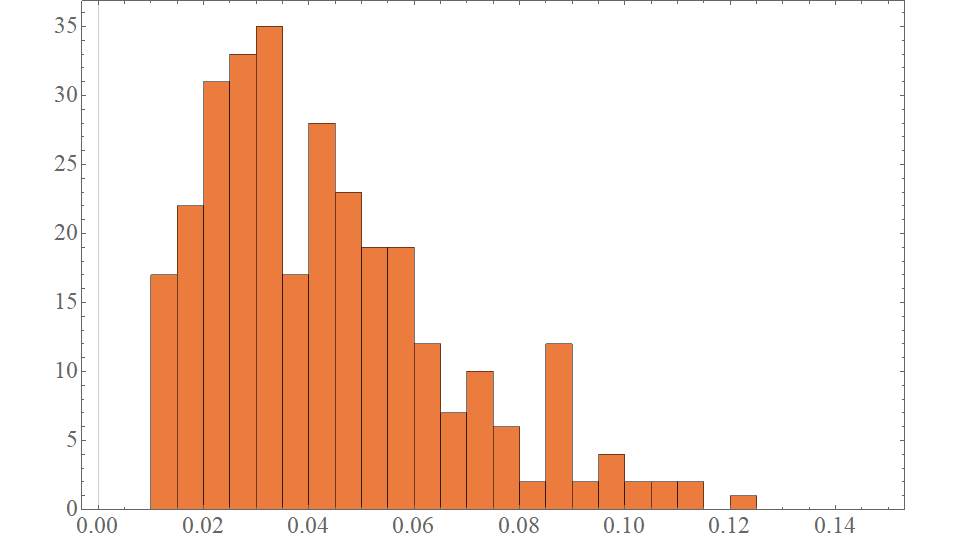} \includegraphics[width=0.5\textwidth]{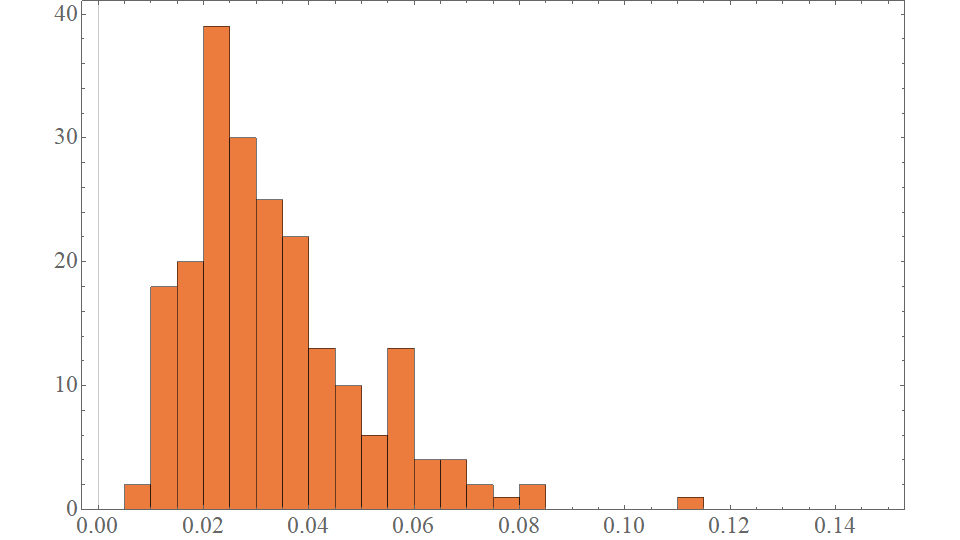}\\Histogram of the distance from $\lambda$ for a selected subset of most significant orbits in $F_{580615}$ (306 orbits) and $F_{609427}$(212 orbits).\hrule\smallskip
Observe that for most larger denominators there appears a small isolated cluster above 0.10. Thsese are the outliers.

\pagebreak
Below we present some other outlier orbits corresponding to denominators larger than $10,000$:\\\label{moreorbits_start}
\smallskip\hrule\smallskip\noindent\includegraphics[width=0.5\textwidth]{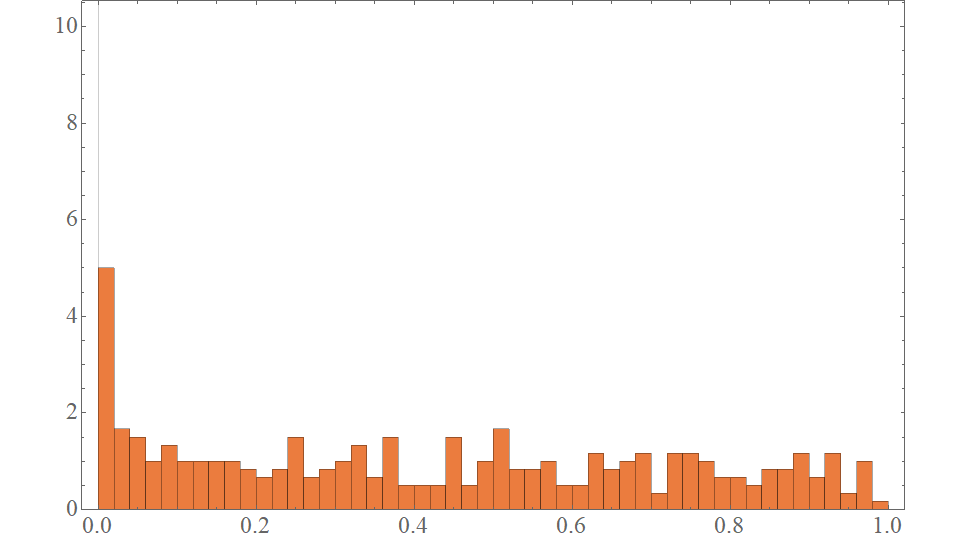} \includegraphics[width=0.5\textwidth]{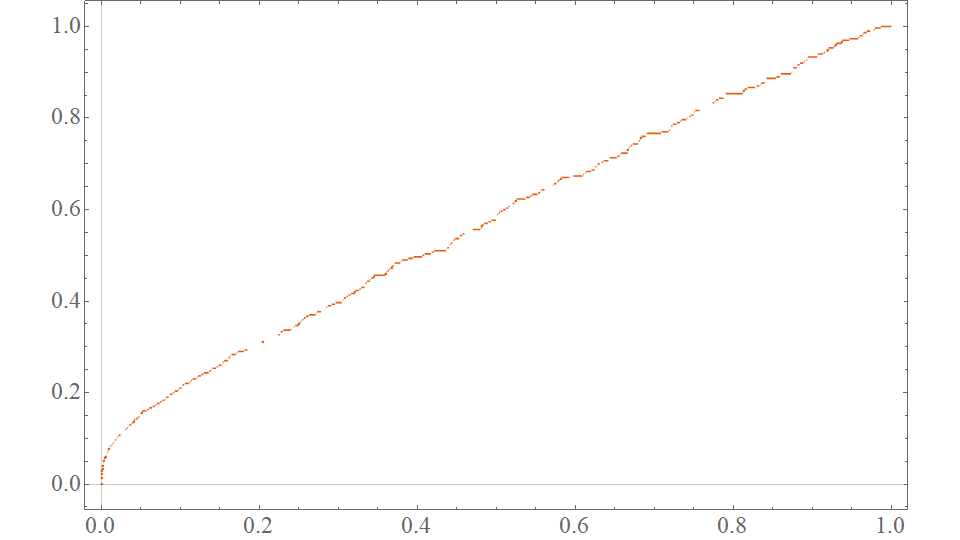}\\1. The histogram and CDF for the orbit of $\frac{1}{15025}$. The orbit has 300 points.\smallskip\hrule\smallskip\noindent\includegraphics[width=0.5\textwidth]{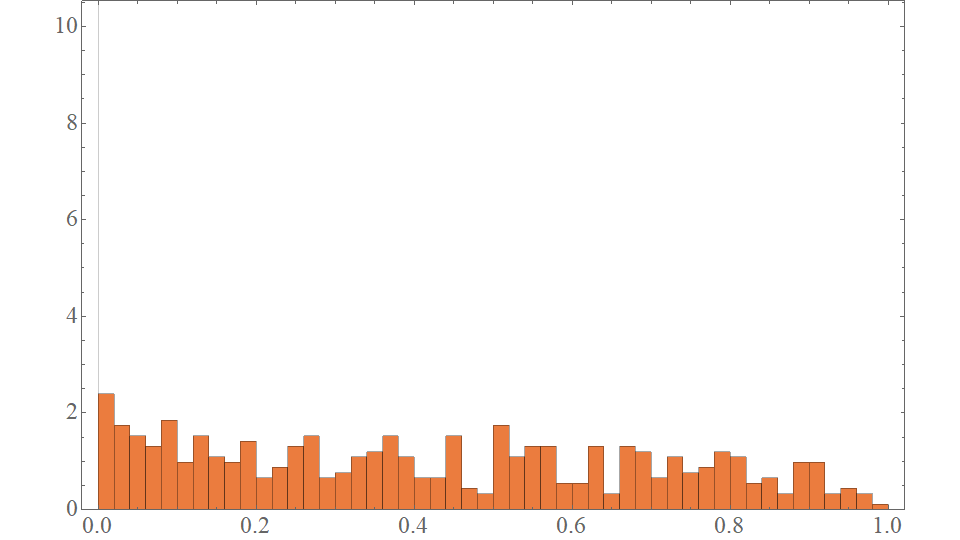} \includegraphics[width=0.5\textwidth]{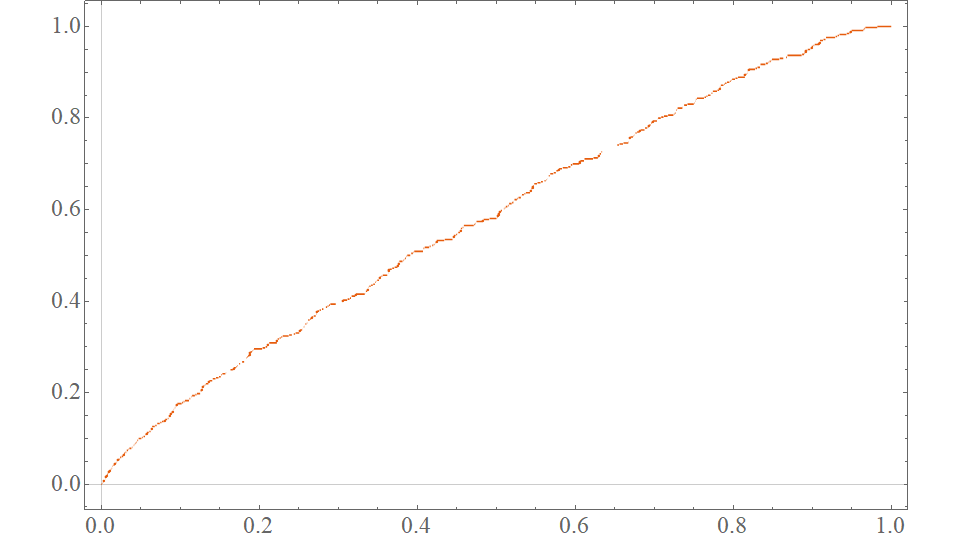}\\2. The histogram and CDF for the orbit of $\frac{23}{21667}$. The orbit has 460 points.\smallskip\hrule\smallskip\noindent\includegraphics[width=0.5\textwidth]{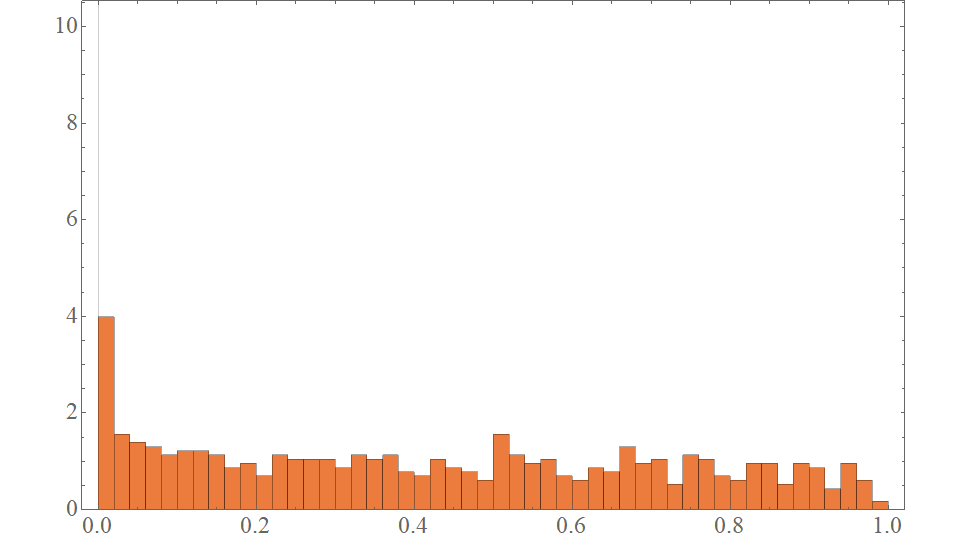} \includegraphics[width=0.5\textwidth]{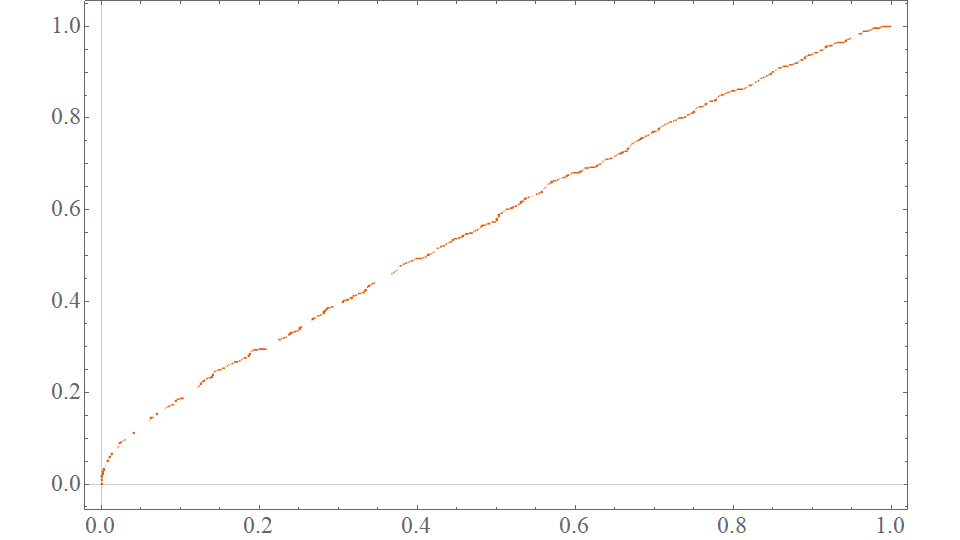}\\3. The histogram and CDF for the orbit of $\frac{1}{22015}$. The orbit has 576 points.\smallskip\hrule\smallskip\noindent\includegraphics[width=0.5\textwidth]{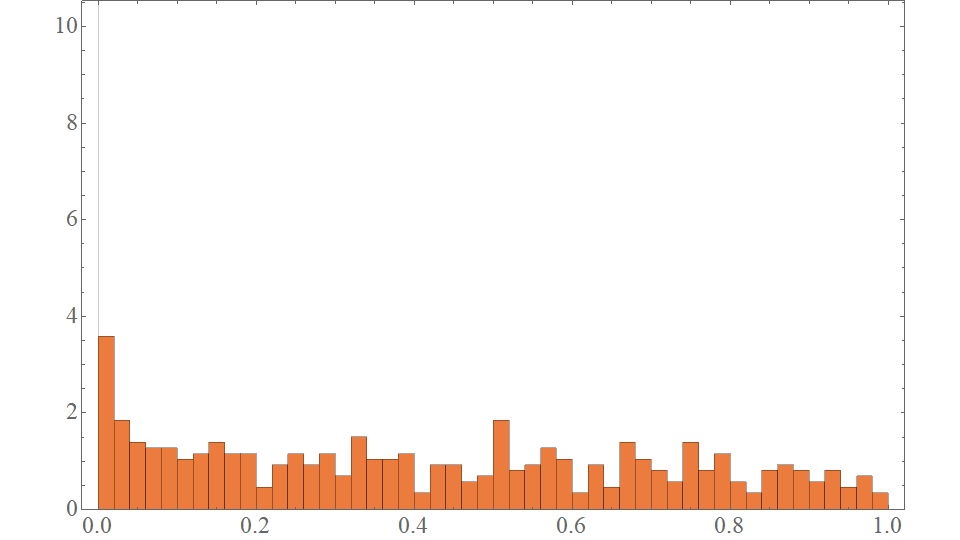} \includegraphics[width=0.5\textwidth]{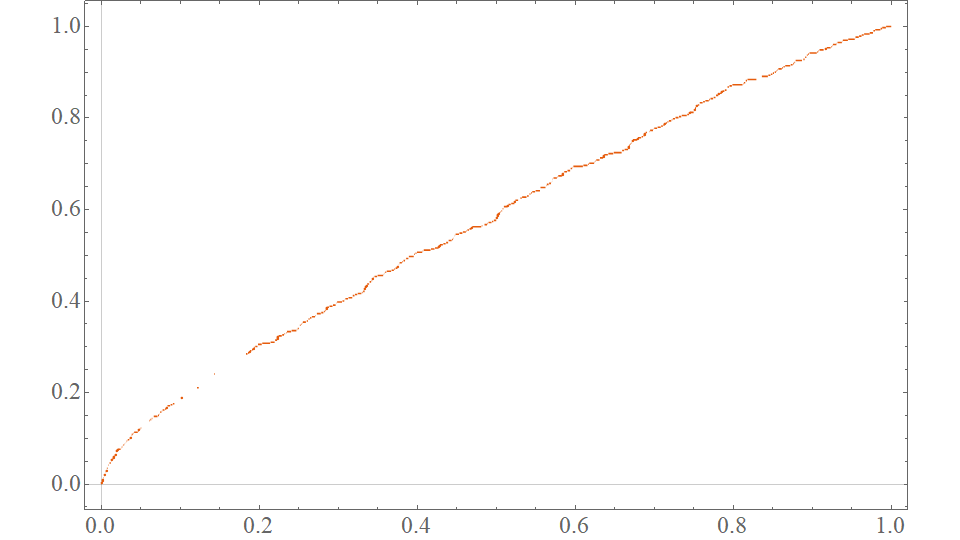}\\4.The histogram and CDF for the orbit of $\frac{19}{33215}$. The orbit has 432 points.\smallskip\hrule\smallskip\noindent \includegraphics[width=0.5\textwidth]{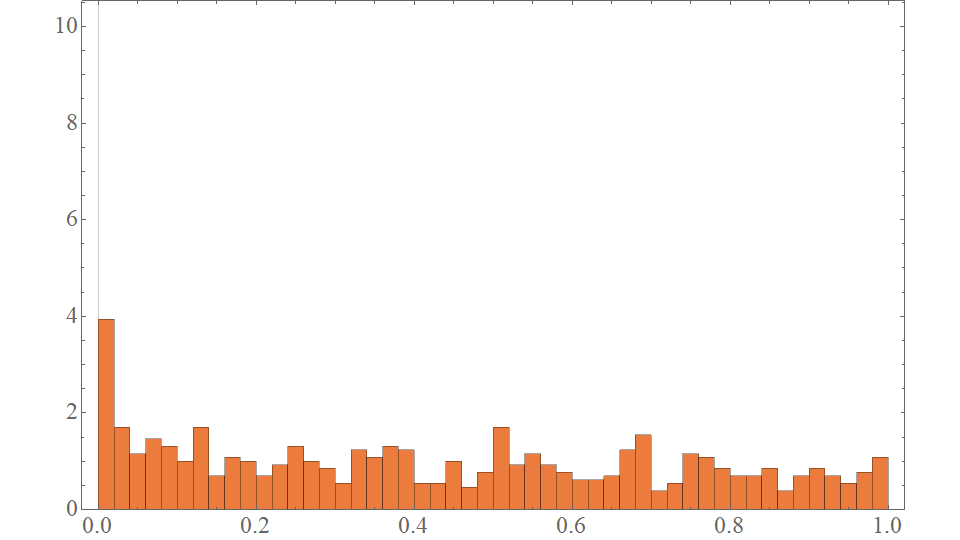}\noindent\includegraphics[width=0.5\textwidth]{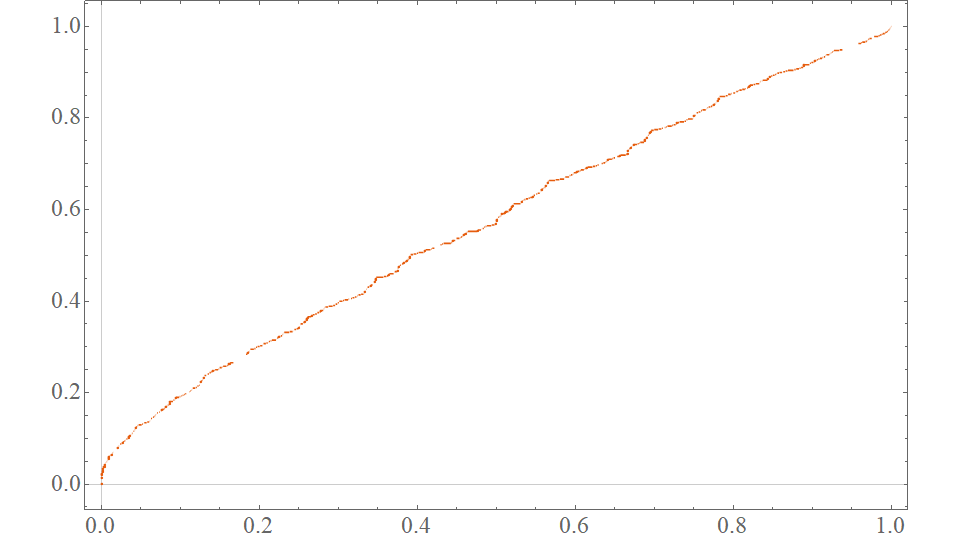}\\5. The histogram and CDF for the orbit of $\frac{1}{50557}$. The orbit has 648 points.\smallskip\hrule\smallskip \includegraphics[width=0.5\textwidth]{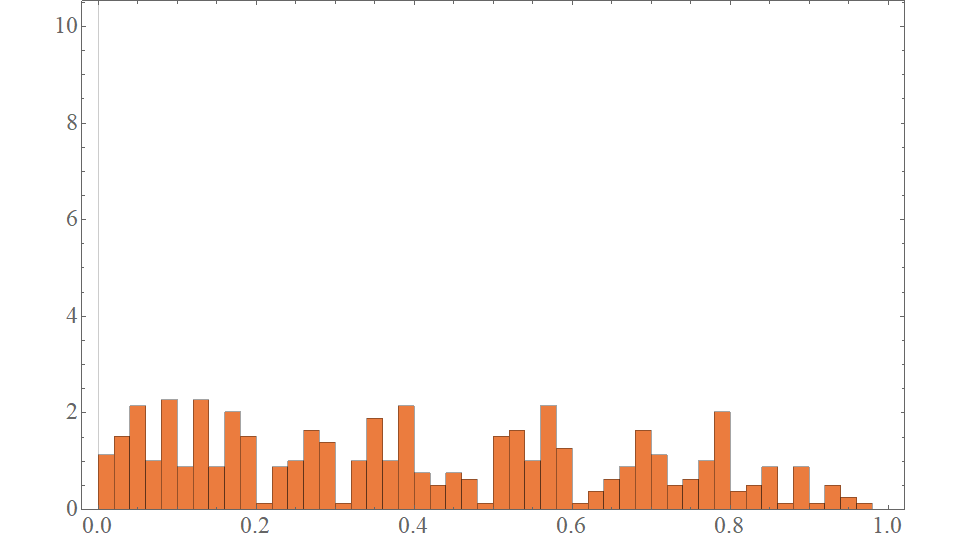}\noindent\includegraphics[width=0.5\textwidth]{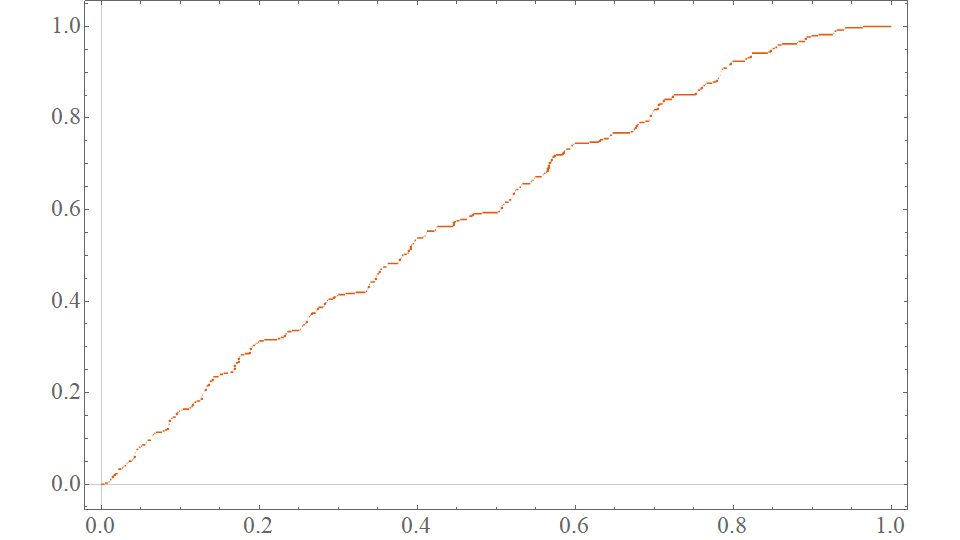}\\6. The histogram and CDF for the orbit of $\frac{373}{86963}$. The orbit has 396 points.\smallskip\hrule\smallskip \includegraphics[width=0.5\textwidth]{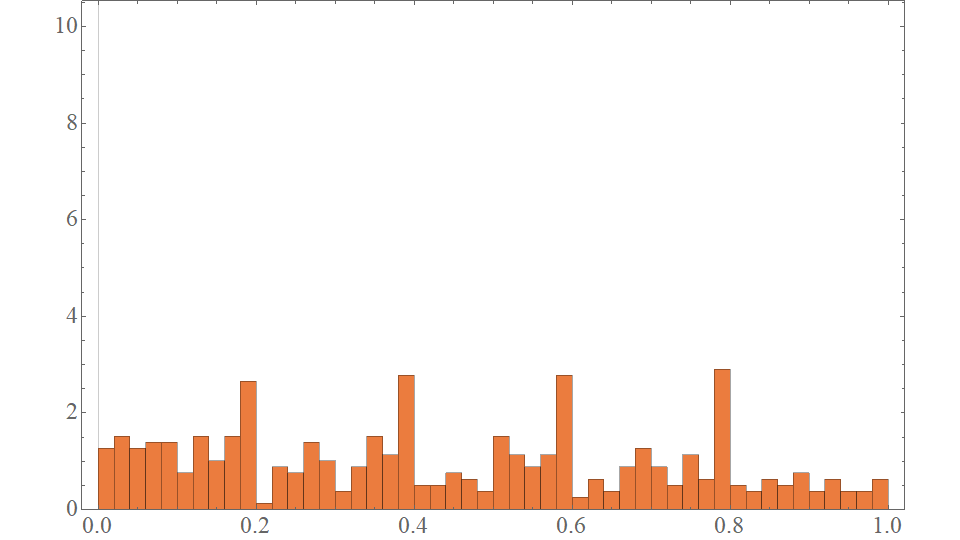}\noindent\includegraphics[width=0.5\textwidth]{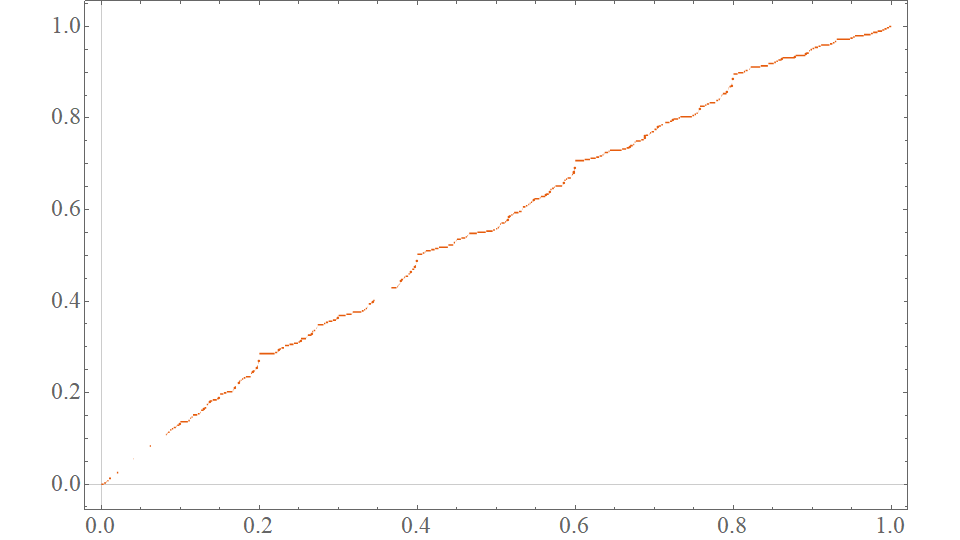}\\7. The histogram and CDF for the orbit of $\frac{271}{86963}$. The orbit has 396 points.\smallskip\hrule\smallskip  \includegraphics[width=0.5\textwidth]{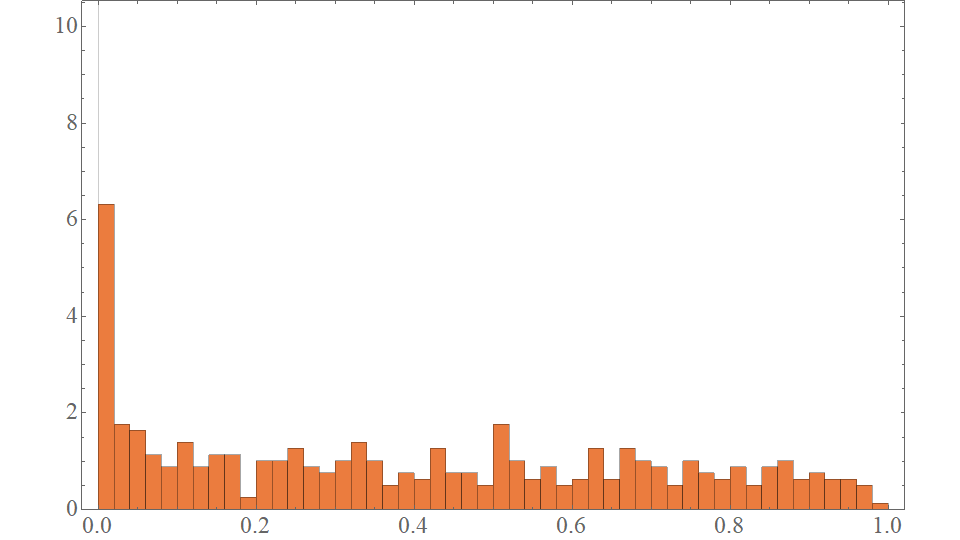}\noindent\includegraphics[width=0.5\textwidth]{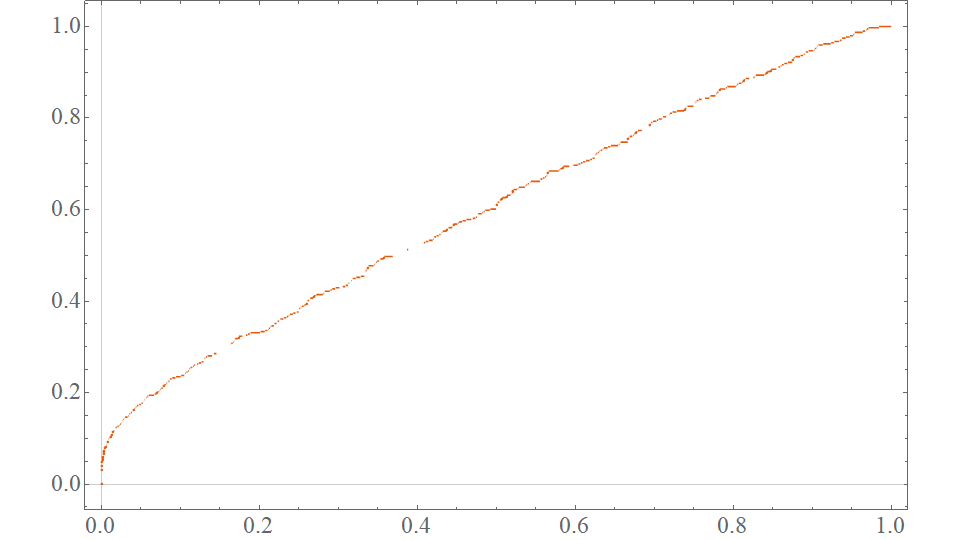}\\8. The histogram and CDF for the orbit of $\frac{1}{86963}$. The orbit has 396 points.\smallskip\hrule\smallskip \includegraphics[width=0.5\textwidth]{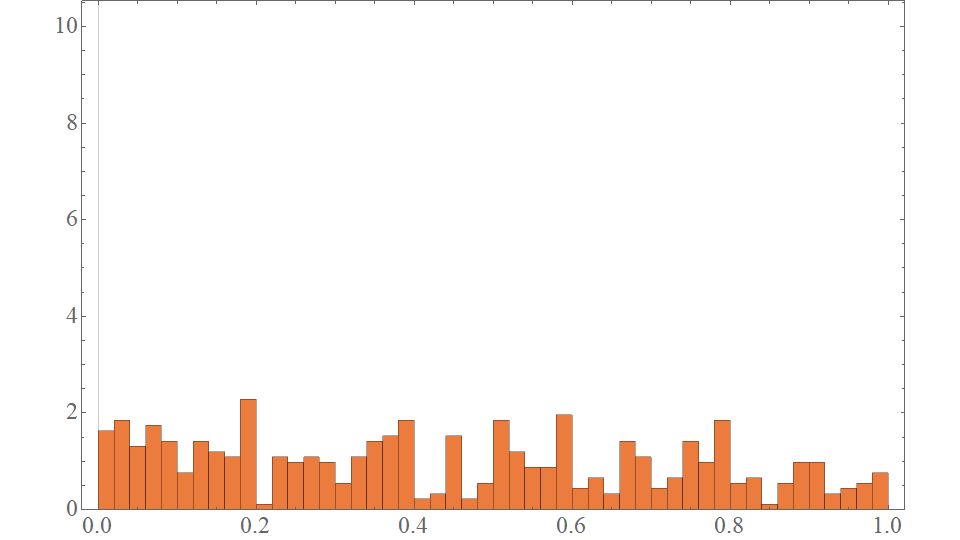}\noindent\includegraphics[width=0.5\textwidth]{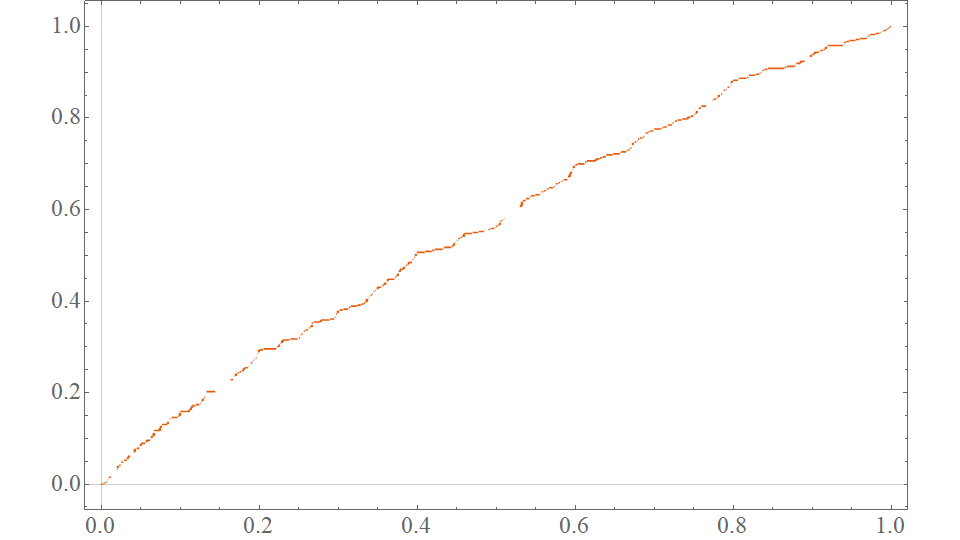}\\9. The histogram and CDF for the orbit of $\frac{259}{108335}$. The orbit has 460 points.\smallskip\hrule\smallskip  \includegraphics[width=0.5\textwidth]{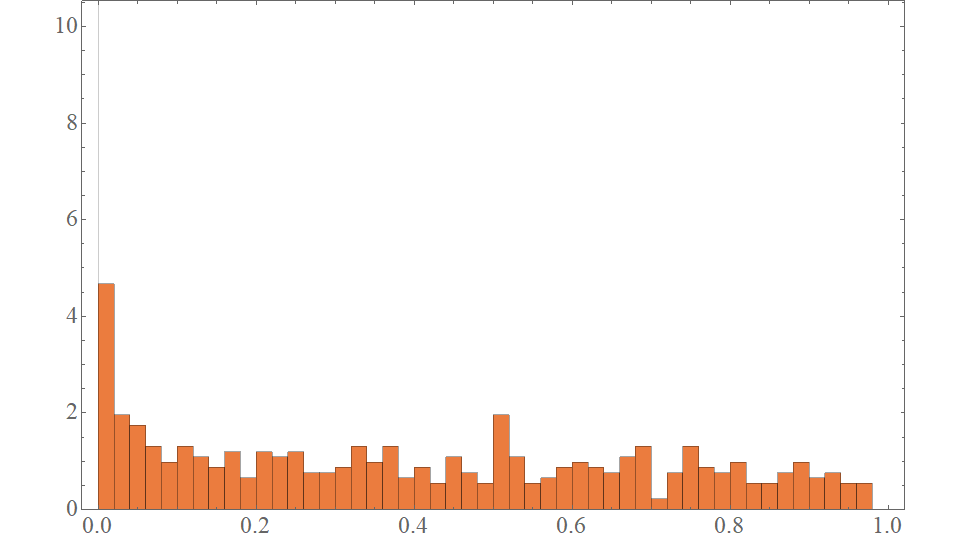}\noindent\includegraphics[width=0.5\textwidth]{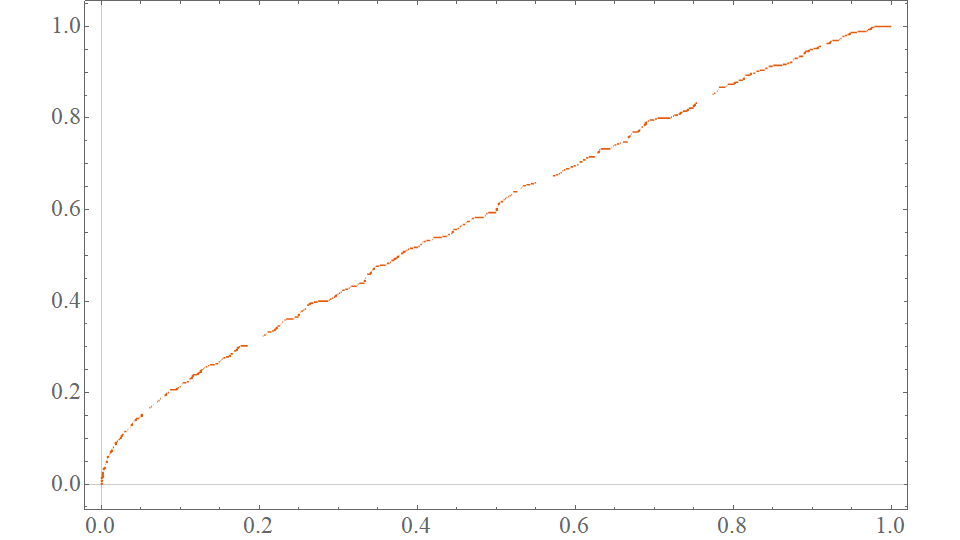}\\10. The histogram and CDF for the orbit of $\frac{29}{108335}$. The orbit has 460 points.\smallskip\hrule\smallskip \includegraphics[width=0.5\textwidth]{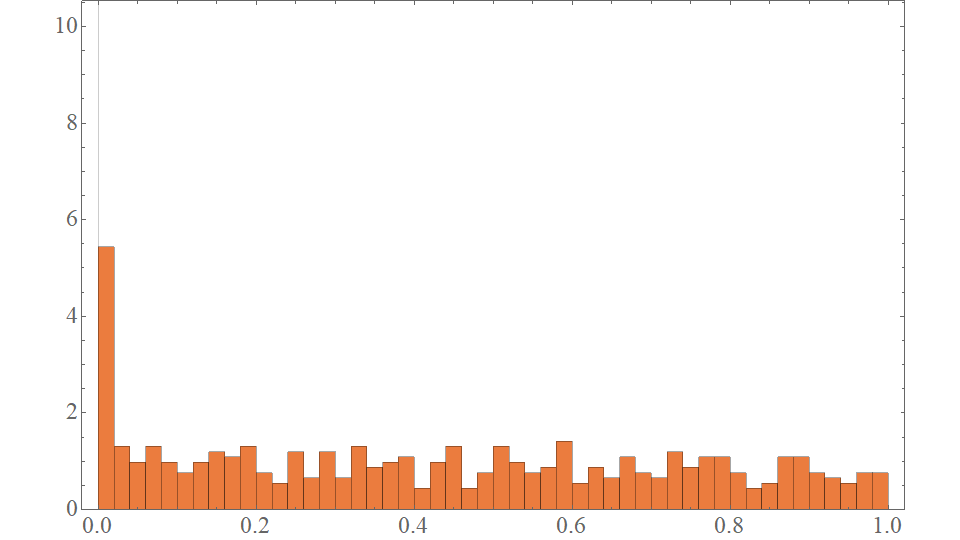}\noindent\includegraphics[width=0.5\textwidth]{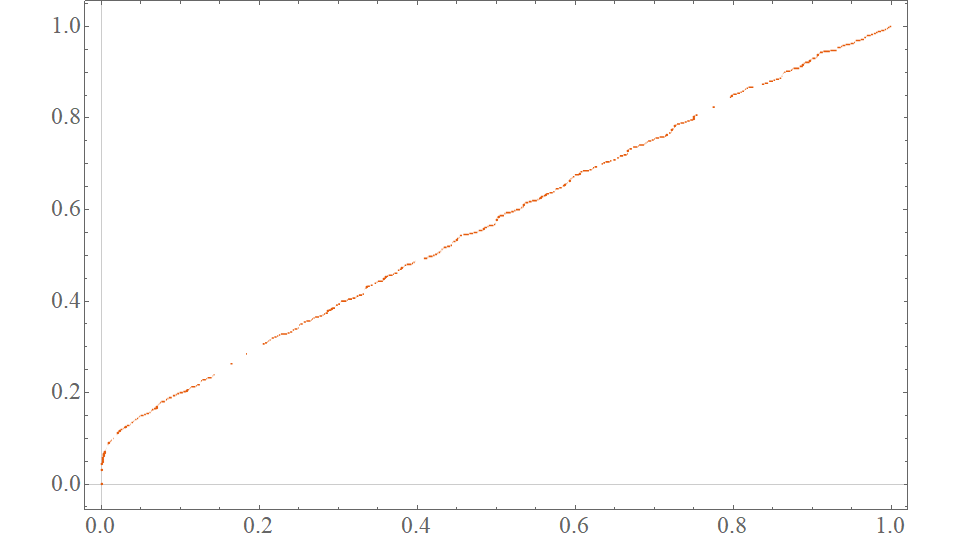}\\11. The histogram and CDF for the orbit of $\frac{1}{108335}$. The orbit has 460 points.\smallskip\hrule\smallskip \includegraphics[width=0.5\textwidth]{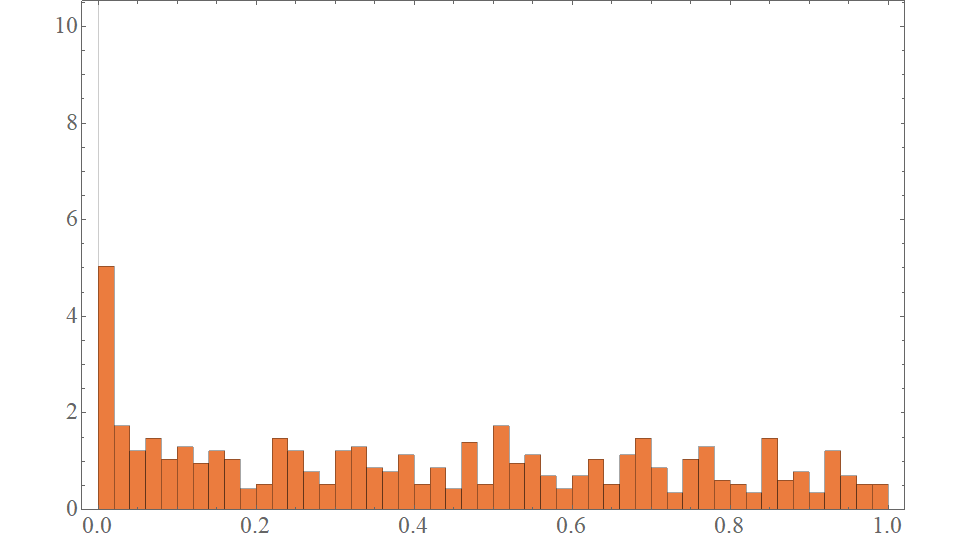}\noindent\includegraphics[width=0.5\textwidth]{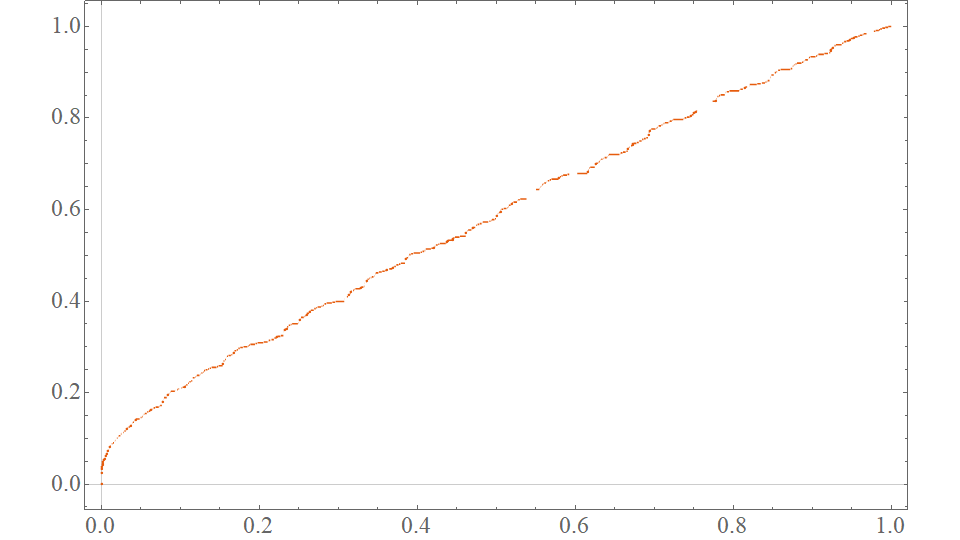}\\12. The histogram and CDF for the orbit of $\frac{1}{119795}$. The orbit has 576 points.\smallskip\hrule\smallskip \includegraphics[width=0.5\textwidth]{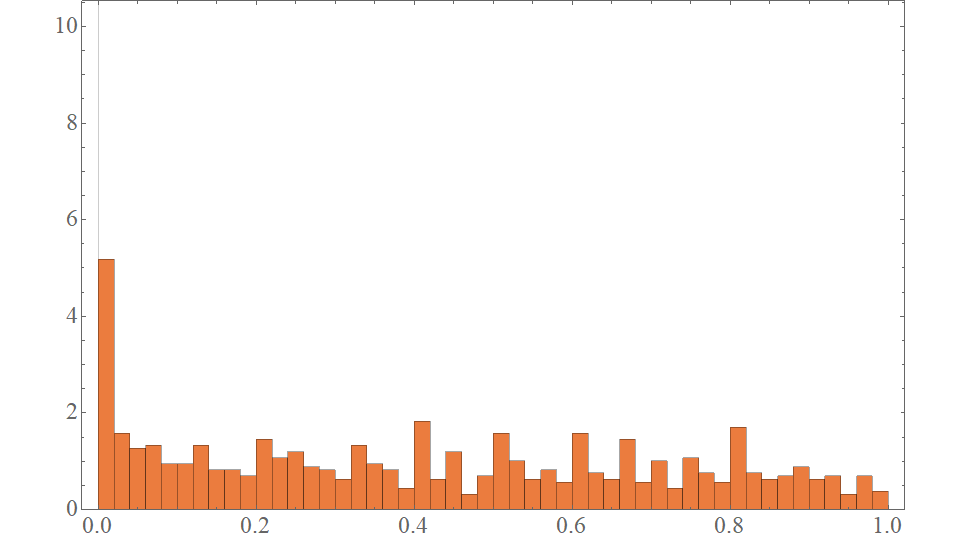}\noindent\includegraphics[width=0.5\textwidth]{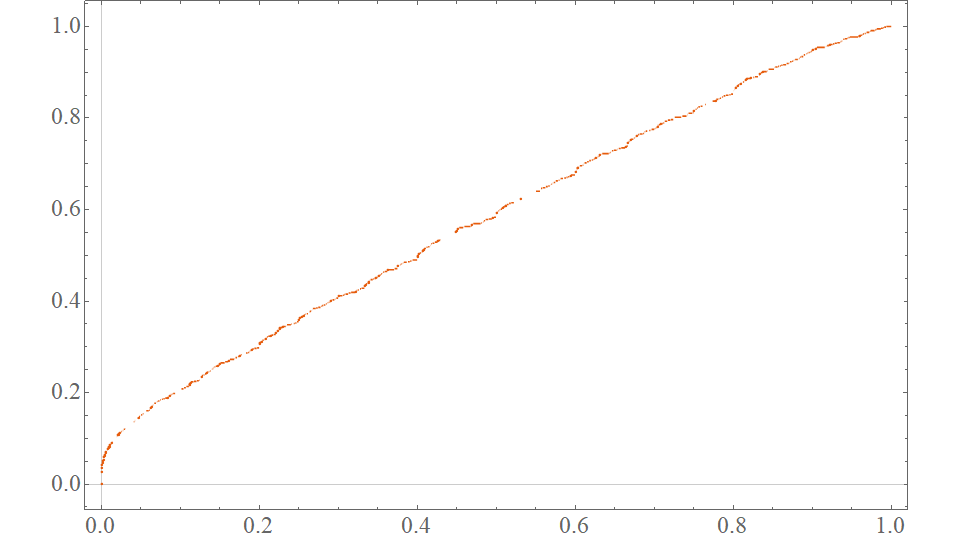}\\13. The histogram and CDF for the orbit of $\frac{1}{434815}$. The orbit has 792 points.\smallskip\hrule\smallskip \includegraphics[width=0.5\textwidth]{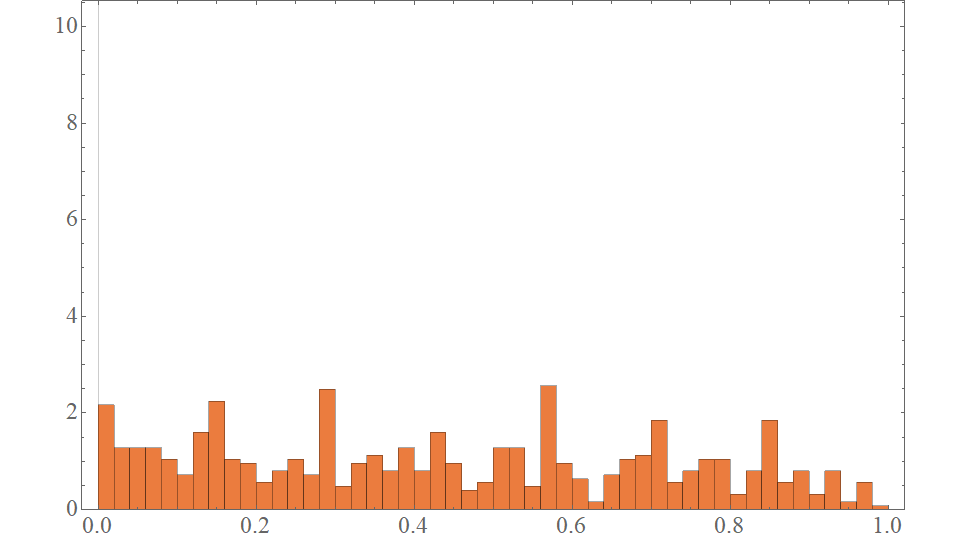}\noindent\includegraphics[width=0.5\textwidth]{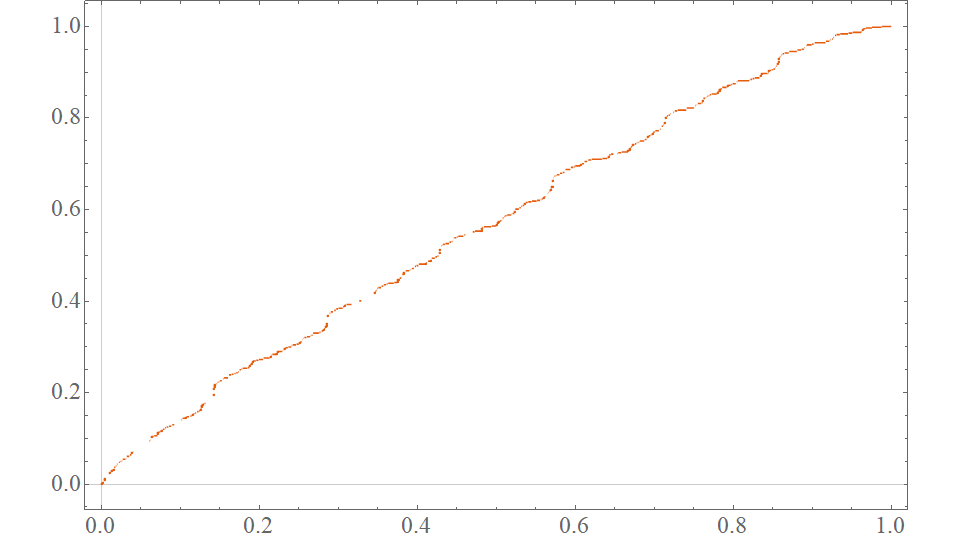}\\14. The histogram and CDF for the orbit of $\frac{289}{580615}$. The orbit has 624 points.\smallskip\hrule\smallskip \includegraphics[width=0.5\textwidth]{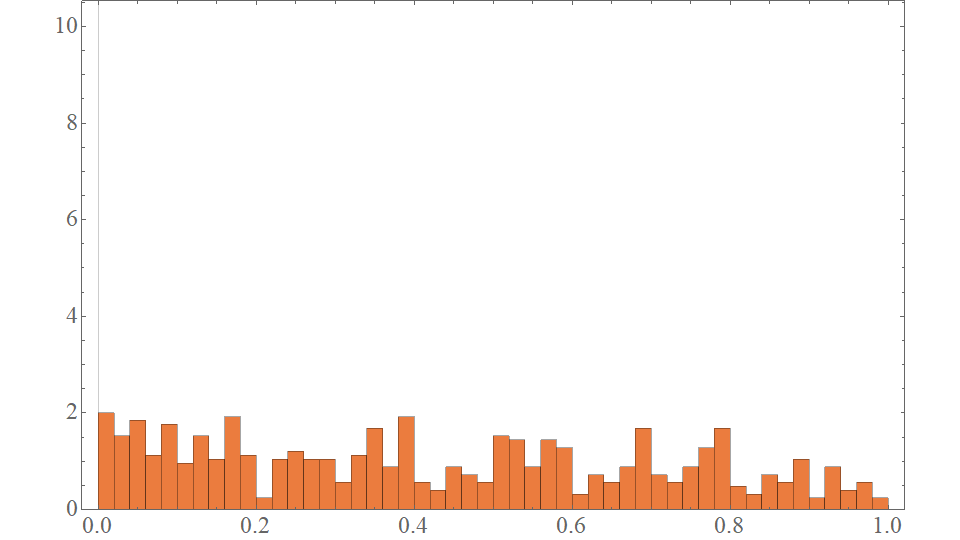}\noindent\includegraphics[width=0.5\textwidth]{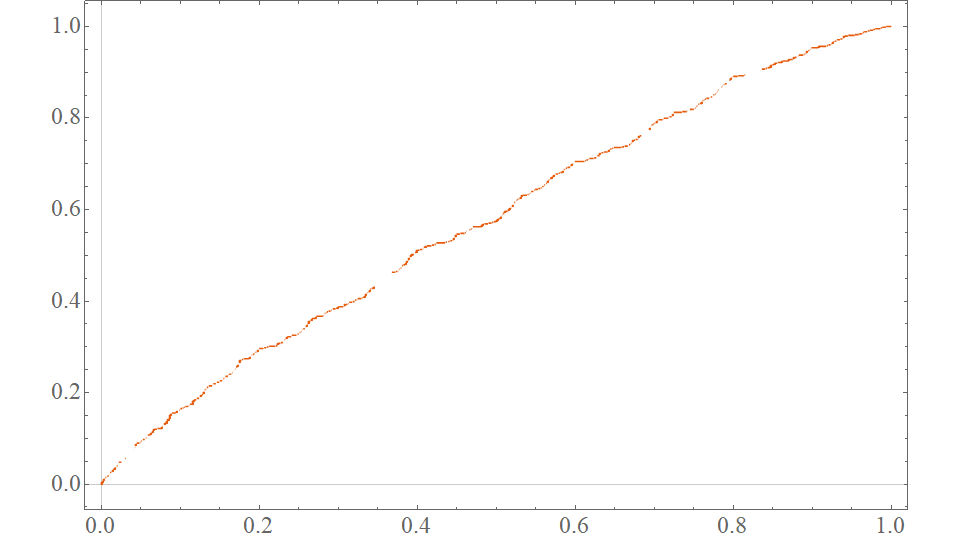}\\15. The histogram and CDF for the orbit of $\frac{233}{580615}$. The orbit has 624 points.\smallskip\hrule\smallskip \includegraphics[width=0.5\textwidth]{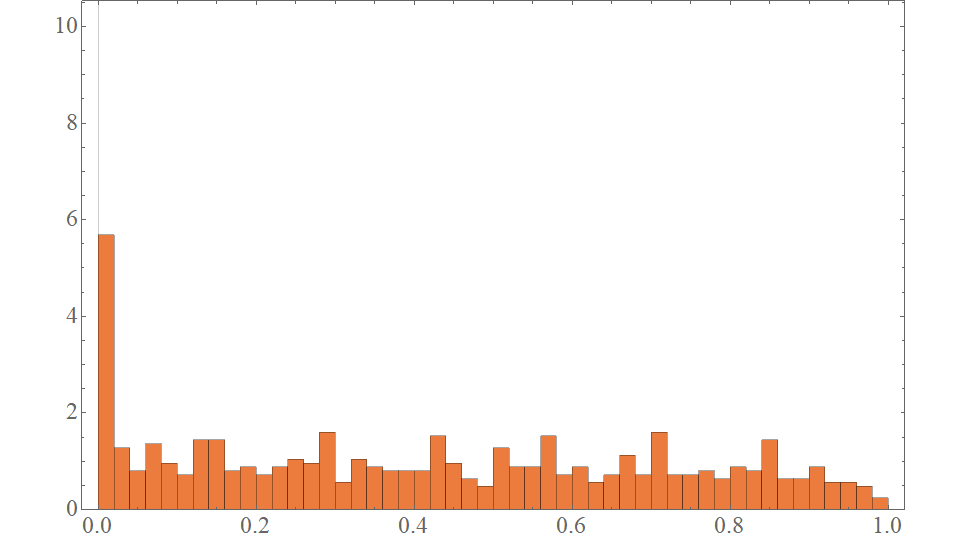}\noindent\includegraphics[width=0.5\textwidth]{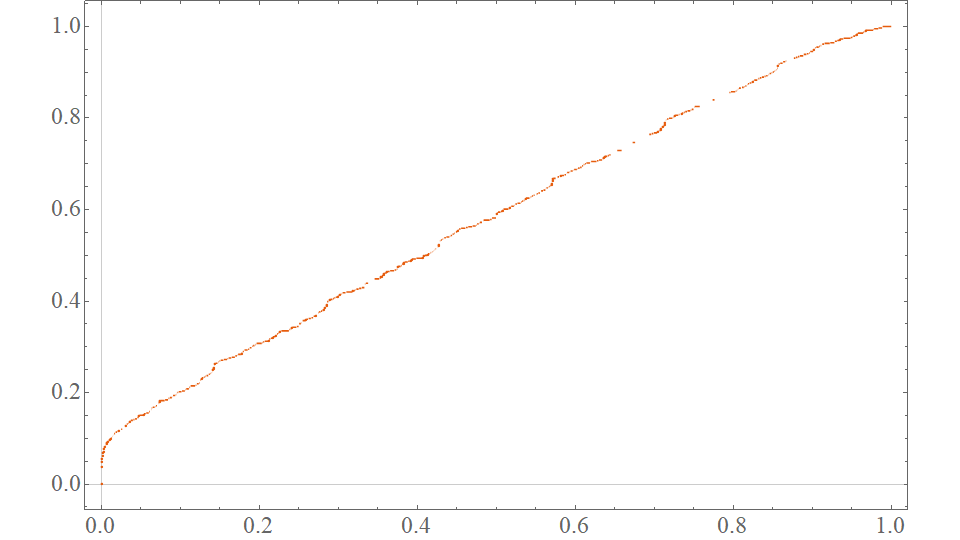}\\16. The histogram and CDF for the orbit of $\frac{1}{580615}$. The orbit has 624 points.\smallskip\hrule\smallskip \label{moreorbits_end}
\pagebreak
Note that many of the outlier orbits have a significant concentration of mass near $0$, which suggests the possibility of converging to a measure with an atom at $0$. Furthermore, several outlier orbits have concentrations of mass near elements of short periodic orbits: the orbit of $\frac{271}{86963}$ in figure 7 above has visible concentrations around the points $\frac{1}{5},\frac{2}{5},\frac{3}{5},$ and $\frac{4}{5}$. Furthermore, in accordance with the reasoning in the proof of theorem \ref{fct:atom_at_one}, applying the transformation $x\mapsto 5x\mod 1$ to the orbit in figure 7 yields an orbit with a concentration of mass near $0$ --- in this case it is specifically the orbit of the point $\frac{1}{86963}$ shown in figure $8$. A similar example is shown in figure $14$, where the orbit has concentrations of mass near points belonging to the orbit of $\frac{1}{7}$. Again, applying the transformation $x\mapsto 7x\mod 1$ to the orbit in figure $14$ ``moves'' these concentrations to the vicinity of $0$, however in this case the resulting orbit is the orbit of $\frac{1}{82945}$ whose atomic measure is not far enough from the Lebesgue measure $\lambda$ (the distance is $\frac{1986949}{25878840}\approx 0.077$) to consider it one of the outliers by our criteria. The decrease in distance may appear surprising, since the new orbit should have a concentration of mass near $0$: this concentration does in fact appear, but the orbit of $\frac{1}{82945}$, unlike that of $\frac{289}{580615}$, is symmetric around $\frac{1}{2}$, and as a result the distance between its atomic measure and $\lambda$ is smaller, as seen in the following images:
\smallskip\hrule\smallskip\noindent\includegraphics[width=0.5\textwidth]{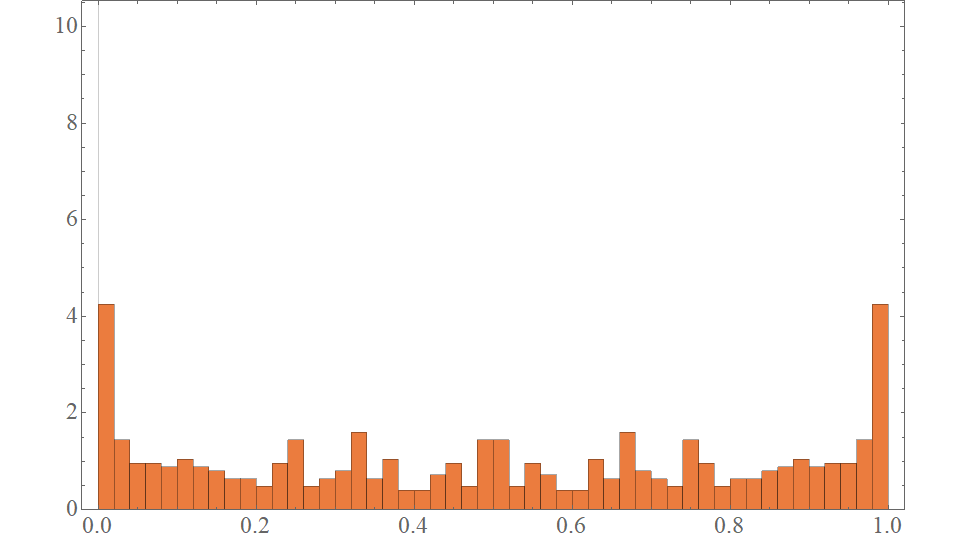} \includegraphics[width=0.5\textwidth]{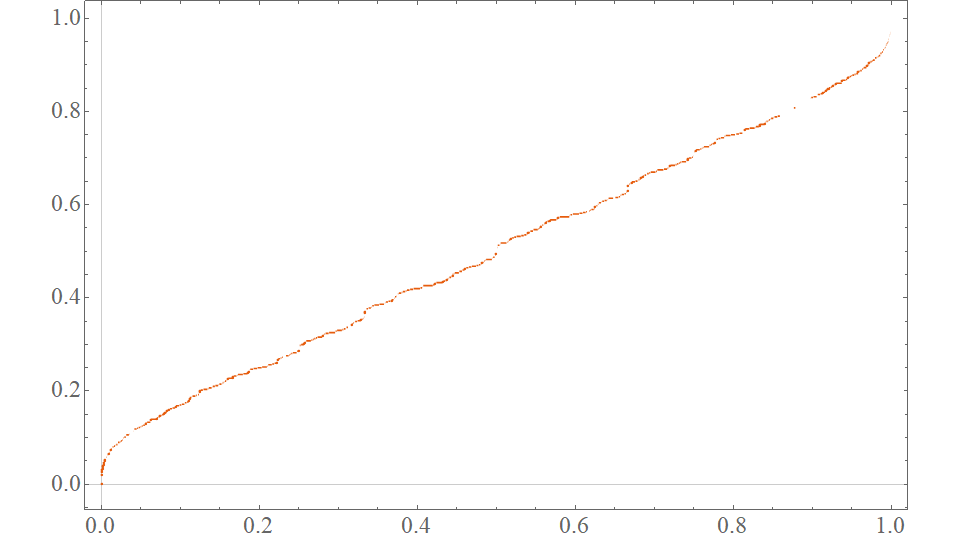}\\The histogram and CDF for the orbit of $\frac{1}{82945}$.\smallskip\hrule\smallskip

\subsection{Another way to view the orbits}
For each of our $32$ chosen outlier orbits we also generated two-dimensional bitmap images based on the orbit of a point $x$ closest to $0$ in the orbit. For a given $x$, the first row of the bitmap is the binary expansion of $x$ (with $0$ corresponding to a black pixel, and $1$ corresponding to a white pixel), the second row is the binary expansion of $3x$, and so on. In other words, for a given starting point $x$ (which we always took to be the smallest element of a given orbit), the pixel in the $i$-th row and $j-$th column (counting from the top left corner) is colored white or black depending on whether the point $3^i2^jx \mod 1$ is smaller or greater than $\frac{1}{2}$, respectively. In the resulting binay representation, the transformation $x\mapsto 2x \mod 1$ corresponds to the horizontal left shift, while $x\mapsto 3x\mod 1$ is the vertical upward shift. We generated bitmaps of size $128\times 128$ pixels: each such bitmap consists of $16,384$ pixels, which is far more than the cardinality of any orbit we consider, hence all the images exhibit various kinds of periodicity. We begin with a side-by-side comparison of the image corresponding to the orbit of $\frac{271}{86963}$ (which had concentrations of mass around the orbit of $\frac{1}{5}$) and the image corresponding to the orbit of $\frac{1}{5}$:

\begin{center}
\includegraphics[height=0.5\textwidth]{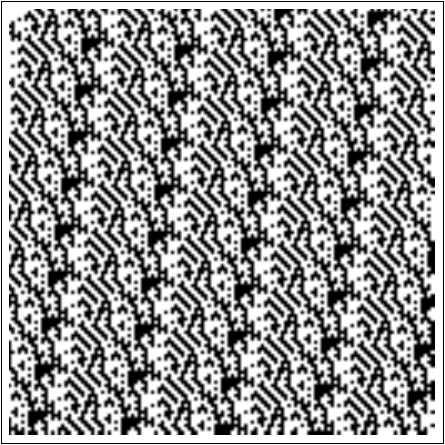}\includegraphics[height=0.5\textwidth]{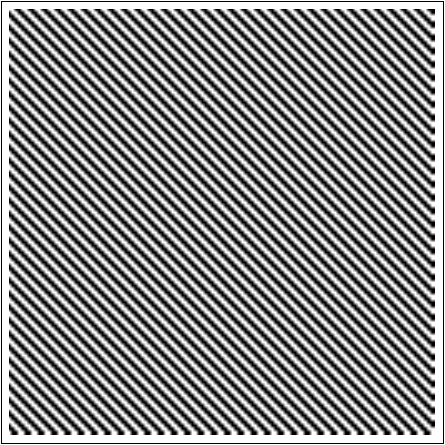}
\end{center}
 Observe that the left image has areas where the pattern is exactly the same as in the right image, which corresponds to the fact that the orbit of $\frac{271}{86963}$ at certain times closely shadows the orbit of $\frac{1}{5}$.
A similar phenomenon can be observed by comparing the symbolic visualizations of the orbit of $\frac{289}{580615}$ and the orbit of $\frac{1}{7}$:\\
\includegraphics[width=0.5\textwidth]{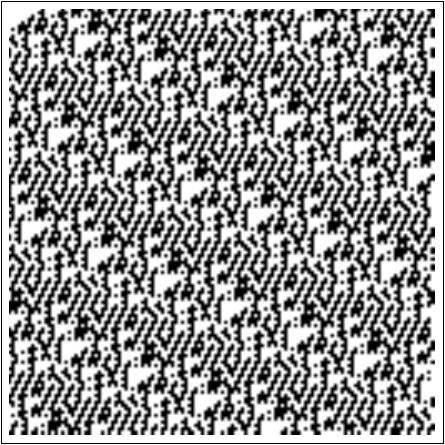} \includegraphics[width=0.5\textwidth]{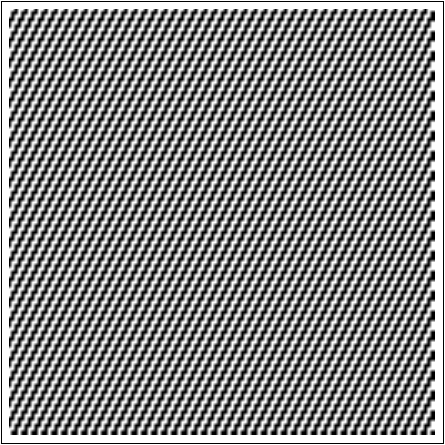}\\
Again, the more complex pattern corresponding to the longer orbit includes sections where it is identical to the pattern corresponding to the shorter orbit, i.e. the orbit of $\frac{233}{580615}$ at times shadows the orbit of $\frac{1}{7}$. 
Below we present the bitmap visualizations of several other outlier orbits:\\
\includegraphics[width=0.5\textwidth]{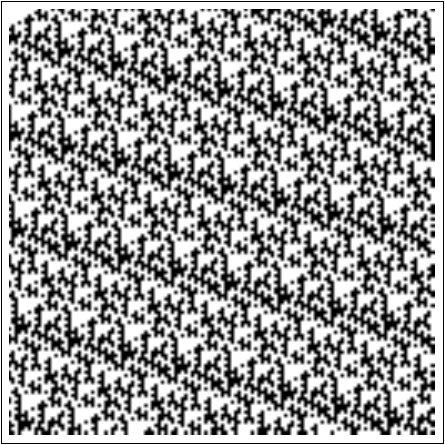}\includegraphics[width=0.5\textwidth]{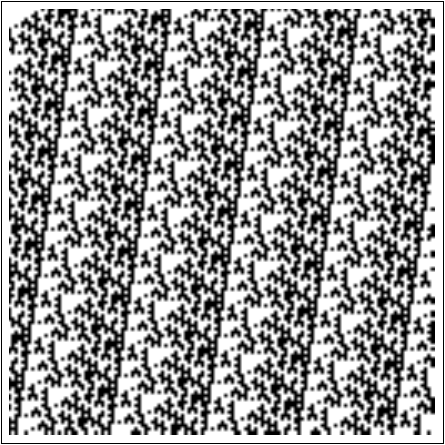}\\
Symbolic visualisations of the orbits of $\frac{23}{21667}$ and $\frac{61}{86963}$.\smallskip\hrule\smallskip\noindent
\includegraphics[width=0.5\textwidth]{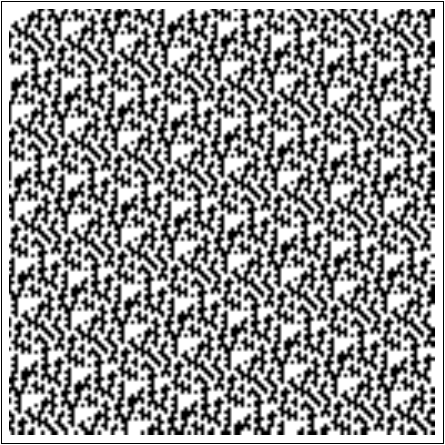}\includegraphics[width=0.5\textwidth]{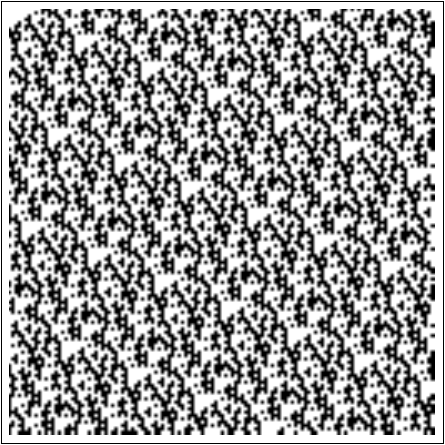}\\
Symbolic visualisations of the orbits of $\frac{179}{108335}$ and $\frac{103}{116123}$.\smallskip\hrule\smallskip\noindent
\includegraphics[width=0.5\textwidth]{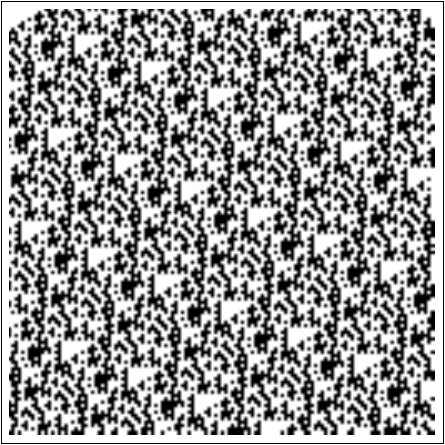}\includegraphics[width=0.5\textwidth]{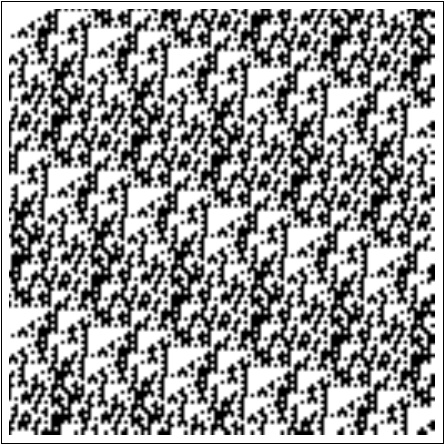}
Symbolic visualisations of the orbits of $\frac{233}{580615}$ and $\frac{31}{609427}$.\smallskip\hrule\smallskip
Since we specifically chose the outlier orbits which have more points in the left half of $[0,1]$, the images have more white than black pixels. Furthermore, the triangular white areas correspond to times where $2^j3^jx\mod 1$ remains in the left half of $[0,1]$ for several consecutive values of $i$ and $j$. This is only possible when the orbit includes points very close to $0$, which corresponds to the atomic measure having a concentration of mass near $0$. Observe that images with larger area of white triangles correspond to orbits with larger concentrations of mass near $0$: if we denote by $x$ the starting point of the orbit, then a white triangle with the upper left corner in row $n$ and column $m$, whose horizontal side has $a$ pixels and vertical side has $b$ pixels, occurs when the point $y=3^n2^mx \mod 1$ is so close to $0$ that $2y,4y,\ldots,2^ay$ are all in $[0,\frac{1}{2}]$, and the same is true of $3y,9y,\ldots,3^by$, which implies that the same is also true for numbers of the form $2^i3^jy$ for $i$ and $j$ up to certain values. Also observe that in fact every row of the image is determined by the preceding one, since row $i$ of the image can be interpreted as the fractional part of the binary expansion of $3^{i-1}x\mod 1$, where $x$ is the starting point of the orbit. It follows that every row can be used to determine the next one by adding it to its own copy shifted to the left by one pixel (this process corresponds to adding $x$ to $2x$ modulo 1, thus producing $3x\mod 1$). 

Finally, we present two other bitmap visualizations: our longest outlier orbit, that of $\frac{31}{609427}$ (left), and the orbit of $\frac{1}{580615}$ (right).\\\\
\label{fig:SymbolicLongest}\includegraphics[width=0.5\textwidth]{symImage032.pdf}
\label{fig:SymbolicLongest}\includegraphics[width=0.5\textwidth]{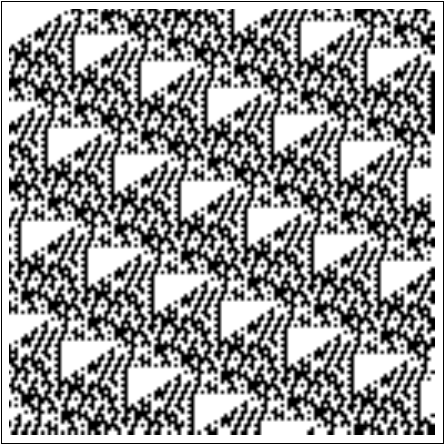}
\smallskip
Since the latter orbit includes the point $\frac{1}{580615}$, the presence of large white triangles is not surprising. Also observe repeated occurrences of the pattern from the bitmap visualization of the orbit of $\frac{1}{7}$, and note that the histogram of this orbit on page \pageref{moreorbits_end} also exhibits concentrations of mass near the points $\frac{1}{7},\ldots,\frac{6}{7}$. In terms of measures this can be seen as indicating the fact that we may hope to find a sequence of atomic measures converging to a measure whose ergodic decomposition will include the atomic measures corresponding to the fixed point $0$ and the periodic orbit of $\frac{1}{7}$.

We shall now explain the presence of ``shadow concentrations'' of mass around $\frac{1}{2}, \frac{1}{3}$ and $\frac23$, whenever a concentration near zero occurs. The point $x$ closest to zero in the orbit generates a white triangle with vertical and horizontal dimensions $a$ and $b$, respectively. Since $a$ is determined as the minimal solution of the inequality $3^ax>\frac12$ and similarly, $b$ is the minimal solution of the inequality $2^bx>\frac12$, there is an approximate relation $b=a\frac{\ln3}{\ln2}$. The area of the white triangle is
$M=a^2\frac{\ln3}{2\ln2}$ which represents the number of units of (two-dimensional) time spent by the orbit near zero. The left margin of the triangle consists of black pixels, each followed to the right by several white pixels. This corresponds to the binary expansion of points near $\frac12$. So, the orbit spends $a$ units of time near $\frac12$. The top margin of the triangle is always a periodic pattern with alternating black and white pixels, corresponding to binary expansions of points close to, alternately, $\frac13$ and $\frac23$. This explains the shadow concentrations with proportions $b=\sqrt{M\frac{2\ln2}{\ln3}}$ points close to $\frac12$ and $\frac a2=\frac12\sqrt{M\frac{2\ln3}{\ln2}}$ points close to each of the points $\frac13$ and $\frac23$. Arguing similarly, we can derive ``second order'' shadow concentrations near $\frac14, \frac34$ and likewise $\frac19, \frac29,\frac49,\frac59,\frac79,\frac89$. Because no invariant atomic measure has atoms with denominators of the form $2^i3^j$, these shadow concentrations must vanish relatively to the length $n$ of the orbit as $n$ increases, and apparently they do.
\section{Acknowledgments}
Research of both authors is supported from:
\begin{itemize}
    \item Resources for science in years
    2013-2018 as research project (NCN grant 2013/08/A/ST1/00275, Poland
    \item Statutory research funds of Faculty of Pure and Applied Mathematics at Wroc\l aw University of Science and Technology
\end{itemize}

\end{document}